\documentclass[12pt,twoside]{amsart}

\usepackage{amsmath}
\usepackage{amssymb}
\usepackage{amscd}
\usepackage{graphics}
\usepackage{graphicx}
\usepackage{epstopdf}
\usepackage{epsfig}
\usepackage{tikz}

\renewcommand{\bar}{\overline}

\newcommand{\eps}{\varepsilon}

\newcommand{\CC}{\mathbb{C}}

\newcommand{\FF}{\mathbb{F}}

\newcommand{\LL}{\mathbb{L}}
\newcommand{\MM}{\mathbb{M}}

\newcommand{\PP}{\mathbb{P}}
\newcommand{\QQ}{\mathbb{Q}}
\newcommand{\RR}{\mathbb{R}}

\newcommand{\ZZ}{\mathbb{Z}}

\newcommand{\hatL}{\widehat{\LL}}

\newcommand{\Qp}{\QQ_p}

\newcommand{\Cp}{\CC_p}

\newcommand{\Fp}{\FF_p}
\newcommand{\Fpbar}{\bar{\FF}_p}

\newcommand{\PCp}{\PP^1(\Cp)}

\newcommand{\calC}{{\mathcal C}}

\newcommand{\calE}{{\mathcal E}}

\newcommand{\calM}{{\mathcal M}}

\newcommand{\calO}{{\mathcal O}}

\newcommand{\Kbar}{\overline{K}}

\newcommand{\PFpbar}{\PP^1(\Fpbar)}

\newcommand{\PKbar}{\PP^1(\overline{K})}

\newcommand{\la}{\langle}
\newcommand{\ra}{\rangle}

\DeclareMathOperator{\Aut}{Aut}
\DeclareMathOperator{\PGL}{PGL}

\DeclareMathOperator{\ord}{ord}

\newcommand{\Dbar}{\bar{D}}

\renewcommand{\mod}{\mbox{mod }}

\theoremstyle{plain}
\newtheorem{thm}{Theorem}[section]

\newtheorem{lemma}[thm]{Lemma}

\theoremstyle{definition}
\newtheorem{defin}[thm]{Definition}

\newtheorem{claim}{Claim}[thm]

\theoremstyle{remark}

\newtheorem{example}[thm]{Example}

\numberwithin{equation}{section}

\title{Discreteness of postcritically finite maps in $p$-adic moduli space}
\date{May 27, 2020}
\subjclass[2010]{11S82, 37P20, 37P45}
\keywords{PCF map, Latt\`{e}s map, Thurston rigidity}
\author{Robert~L. Benedetto}
\address[Benedetto]{Amherst College, Amherst, MA 01002}
\email{rlbenedetto@amherst.edu}
\author{Su-Ion Ih}
\address[Ih]{University of Colorado, Boulder, CO 80309, and
\\ Korea Institute for Advanced Study, Seoul 02455}
\email{ih@math.colorado.edu}

\begin{document}

\begin{abstract}
Let $p \geq 2$ be a prime number and let $\Cp$ be the completion of an algebraic closure 
of the $p$-adic rational field $\Qp$. Let $f_c(z)$ be a one-parameter family of
rational functions of degree $d\geq 2$, where the coefficients are meromorphic functions
defined at all parameters $c$ in some open disk $D\subseteq\Cp$.
Assuming an appropriate stability condition, we prove that the parameters $c$ for which
$f_c$ is postcritically finite (PCF) are isolated from one another in the $p$-adic disk $D$,
except in certain trivial cases. In particular, all PCF parameters of the family
$f_c(z)=z^d+c$ are $p$-adically isolated.
% and $f_c(z)=z^d+cz$
\end{abstract}

\maketitle

%%%%%%%%%%%%%%%%%%%%%%%%%%%%%%%%%%%%%%%%%%%%%%%%%%%%%%%%%%%%

\section{Introduction}
Let $K$ be a field with algebraic closure $\Kbar$,
and let $f(z)\in K(z)$ be a separable rational function.
The degree of $f$ is defined to be the maximum of the degrees
of the numerator and denominator of $f$.
Define $f^n(z)$ to be the $n$-th iterate of $f$ under composition;
that is, $f^0(z)=z$, and for each $n\geq 0$, $f^{n+1}(z)=f\circ f^n(z)$.
The \emph{forward orbit} of a point $x\in\PKbar$ is the set
$\{f^n(x) : n\geq 0\}$; the \emph{strict forward orbit} of $x$ is $\{f^n(x) : n\geq 1\}$.
We say $x$ is \emph{periodic} under $f$ if $f^n(x)=x$ for some $n\geq 1$;
equivalently, $x$ belongs to its own strict forward orbit.
Similarly, we  say that $x$ is \emph{preperiodic} under $f$ if there are integers
$n>m\geq 0$ such that $f^n(x)=f^m(x)$;
equivalently, the forward orbit of $x$ is finite.
We say $x$ is \emph{strictly preperiodic} if it is preperiodic but not periodic.
We say that $f$ is \emph{postcritically finite}, or PCF, if every
critical point $c\in\PKbar$ of $f$ is preperiodic under $f$.

Let $f_c(z)$ be a one-parameter analytic family of rational functions of degree $d$.
That is, each coefficient of $f_c$ is an analytic function of $c\in D$, where $D$ is some
set of parameters; and for each specific parameter $c\in D$, 
the function $f_c(z)\in K(z)$ is a rational function of degree $d$.
In this paper, we consider the set of parameters $c\in D$ for which $f_c$ is PCF;
we call such points \emph{PCF parameters}.
In particular, our focus is on the case that $K=\Cp$, the completion of an algebraic closure of the
field $\Qp$ of $p$-adic rationals, where $p\geq 2$ is a prime number.

There are a number of cases when it is easy to see that there are many PCF parameters.
First, suppose that $f_c=h_c^{-1} \circ g \circ h_c$,
where $g(z)\in K(z)$ is a (fixed) rational function of degree $d\geq 2$,
and where $h_c(z)$ is a one-parameter family of linear fractional transformations.
This is the \emph{isotrivial} case, and if $g(z)$ itself is PCF,
then so is $f_c$ for every parameter $c$.
Second, suppose that $E_c$ is a one-parameter family of elliptic curves, $m\geq 2$ is an integer,
and $f_c$ is the \emph{flexible Latt\`{e}s map} associated with the multiplication-by-$m$ map
$[m]:E_c\to E_c$.
(That is, if we write $E_c$ in Weierstrass form, then $f_c\circ x_c = x_c\circ [m]$,
where $x_c:E_c\to \PP^1$ is the $x$-coordinate map. See Definition~\ref{def:lattes}
for a slightly more general definition of flexible Latt\`{e}s maps.)
It is well known, and easy to check, that all Latt\`{e}s maps are PCF, and hence
every $c$ is a PCF parameter for such a family.

As a third, more complicated case, let $K=\CC$,
and consider the quadratic polynomial family $f_c(z)=z^2+c\in\CC[c,z]$.
Since the critical point at $z=\infty$ is fixed, a given map $f_c$ is PCF if and only if
the critical point at $z=0$ is preperiodic. The parameters $c$ for which $z=0$
is strictly preperiodic are called \emph{Misiurewicz parameters}, and they
are dense in the boundary $\partial\mathbf{M}$ of the Mandelbrot set $\mathbf{M}$.
(See, for example,  \cite[Theorem~VIII.1.5ff]{CG} or \cite[Proposition~2.1]{Mcm97}.
The parameters $c$ for which $z=0$ is periodic lie in the interior of $\mathbf{M}$
but also accumulate on $\partial\mathbf{M}$.)
%The Misiurewicz parameters are known to be dense in the boundary of $\mathbf{M}$.
%(This follows, for example, from the equidistribution results
%described by Baker and DeMarco \cite{BD};
%see their Corollary~2.10 and~Proposition~3.3(5).
%Alternatively, since $\partial\mathbf{M}$
%is the bifurcation locus for the family $f_c$, then the functions $0\mapsto f_c^n(0)$ do not
%form a normal family, and therefore by Montel's Theorem, they cannot always omit two
%earlier iterates, like $f_c^2(0)$ and $f_c^3(0)$. We thank Holly Krieger for pointing
%out this bifurcation argument.)
Hence, the complex quadratic polynomial family features PCF parameters
accumulating at many points, including one another.
(This phenomenon is not unique to the complex case $K=\CC$.
Indeed, Rivera-Letelier described similar Misiurewicz bifurcations in the $p$-adic setting
%albeit to a different end,
in \cite{Riv5}; see Example~\ref{ex:juan}.)

By contrast, our main result, Theorem~\ref{thm:pcfisol} below,
states that if $K=\Cp$ and the family $f_c$ is stable in an appropriate sense,
then such accumulation of PCF parameters is impossible.
The underlying flavor of this statement is similar to that of
Scanlon's theorem \cite{Sca;98,Sca;99} on the Tate-Voloch conjecture \cite{TV;96},
although our methods are different.
See Section~\ref{ssec:disks} for the notation $D(b,R)$ and $\Dbar(b,R)$ for
open and closed disks in $\Cp$, and Section~\ref{ssec:merom}
for the field $\MM_p(D)$ of $p$-adic
meromorphic functions on a disk $D\subseteq\Cp$;
in addition, $\PP^1(\MM_p(D))$ of course denotes $\MM_p(D)\cup\{\infty\}$.

\begin{thm}
\label{thm:pcfisol}
Let $p\geq 2$ be a prime number, let $S>0$,
%let $\Phi(c,z)\in\Cp(c,z)$,
let $\Phi(c,z)\in\MM_p(D(0,S))(z)$,
and suppose that for each $c\in D(0,S)$,
\[ f_c(z) := \Phi(c,z) \]
is a rational function in $\Cp(z)$ of degree $d\geq 2$.
Let $\alpha_1,\ldots,\alpha_{2d-2}\in\PP^1\big(\MM_p(D(0,S))\big)$
be meromorphic functions on $D(0,S)$ such that for each $c\in D(0,S)$,
the critical points of $f_c$, repeated according to multiplicity,
are $\alpha_1(c),\ldots, \alpha_{2d-2}(c)$.

For each $i=1,\ldots, 2d-2$, suppose that there are
integers $N_i>M_i\geq 0$ and
rational open disks $U_{i,0},\ldots,U_{i,N_i}\subseteq \PCp$
such that for all $c\in D(0,S)$,
\begin{itemize}
\item $\alpha_i(c)\in U_{i,0}$,
\item $f_c(U_{i,j}) \subseteq U_{i,j+1}$ for all $j=0,\ldots, N_i-1$, and
\item $U_{i,N_i}\subseteq U_{i,M_i}$.
\end{itemize}
Then either
\begin{enumerate}
\item for any $0<s<S$, there are
only finitely many $c\in \Dbar(0,s)$ such that
$f_c$ is postcritically finite,
\item for every $c\in D(0,S)$,
$f_c$ is a flexible Latt\`{e}s map, or
\item for every $b,c\in D(0,S)$,
$f_b$ is conjugate to $f_c$.
\end{enumerate}
\end{thm}

The stability conditions of Theorem~\ref{thm:pcfisol}, i.e., the bullet points
concerning the disks $U_{i,j}$, happen to be satisfied by certain families of good reduction,
which we now define. Recall that a non-archimedean field $K_v$
with absolute value $|\cdot|_v$ has ring of integers
$\calO_v:=\{x\in K_v : |x|_v\leq 1\}$, maximal ideal
$\calM_v:=\{x\in K_v : |x|_v< 1\}$, and residue field $k_v:=\calO_v/\calM_v$.
(For $K_v=\Cp$, the residue field is isomorphic to $\Fpbar$, the algebraic closure
of the field $\Fp$ of $p$ elements.)

%We recall that the ring of integers of $\Cp$
%is $\calO_p:=\{x\in\Cp : |x|_p=1\}$, with maximal ideal
%$\calM_p:=\{x\in\Cp : |x|_p<1\}$ and quotient field $\khat_p:=\calO_p/\calM_p$.

\begin{defin}
\label{def:goodred}
Let $K_v$ be a non-archimedean field with absolute value $|\cdot|_v$,
ring of integers $\calO_v$, and residue field $k_v$,
and let $f(z)\in K_v(z)$ be a rational function of degree $d\geq 1$.
Write $f=g/h$, where $g,h\in\calO_v[z]$, with at least one coefficient $a$
of $g$ or $h$ satisfying $|a|_v=1$.
Define $\bar{g},\bar{h}\in k_v[z]$ by reducing each coefficient of $g,h$
modulo $v$. We say $f$ has (\emph{explicit}) \emph{good reduction}
if its reduction, the rational function $\bar{f}:=\bar{g}/\bar{h}\in k_v(z)$,
%satisfies $\deg(\bar{g}/\bar{h})=d$.
satisfies $\deg(\bar{f})=d$.
%Let $p\geq 2$ be a prime number, and let $f(z)\in\Cp(z)$ be a rational function
%of degree $d\geq 2$. Write $f=g/h$ where $g,h\in\calO_p[z]$,
%with at least one coefficient $c$ of $g$ or $h$ satisfying $|c|_p=1$.
%Define $\bar{g},\bar{h}\in\khat_p[z]$ by reducing each coefficient of $g,h$
%modulo $\calM_p$. We say $f$ has (\emph{explicit}) \emph{good reduction}
%if the rational function $\bar{g}/\bar{h}\in\khat_p(z)$ satisfies $\deg(\bar{g}/\bar{h})=d$.
\end{defin}

The reduction map from $\calO_v\to k_v$
induces a map $\PP^1(K_v)\to \PP^1(k_v)$, which we also denote $x\mapsto \bar{x}$.
If $f\in K_v(z)$ has good reduction,
then $\overline{f(x)} = \bar{f} (\bar{x})$ for every $x\in\PP^1(K_v)$.

%Thus, just as reduction modulo the maximal ideal yields a 
%Our claim that maps of good reduction satisfy the stability hypotheses
%of Theorem~\ref{thm:pcfisol} yields the following corollary.

\begin{thm}
\label{thm:goodpcf}
Let $p\geq 2$ be a prime number, let $S>0$,
%let $\Phi(c,z)\in\Cp(c,z)$,
let $\Phi(c,z)\in\MM_p(D(0,S))(z)$,
and let $\alpha_1,\ldots,\alpha_{2d-2}\in\PP^1\big(\MM_p(D(0,S))\big)$.
Suppose that for each $c\in D(0,S)$,
\begin{itemize}
\item
$f_c(z) := \Phi(c,z)$ is a rational function in $\Cp(z)$ of degree $d\geq 2$
and of $p$-adic explicit good reduction;
\item
the reduction $\bar{f}_c\in\Fpbar(z)$ satisfies
$\bar{f}_c = \bar{f}_0$;
\item
the critical points of $f_c$, repeated according to multiplicity,
are $\alpha_1(c),\ldots, \alpha_{2d-2}(c)$; and
\item
for each $i=1,\ldots,2d-2$, the reduction $\overline{\alpha_i(c)}\in\PFpbar$
satisfies $\overline{\alpha_i(c)}=\overline{\alpha_i(0)}$.
\end{itemize}
%Suppose that there are power series
%$\alpha_1,\ldots,\alpha_{2d-2}\in\Cp[[c]]$
%converging on $D(0,S)$ such that for each $c\in D(0,S)$,
%the critical points of $f_c$, repeated according to multiplicity,
%are $\alpha_1(c),\ldots, \alpha_{2d-2}(c)$.
Then either 
\begin{enumerate}
\item for any $0<s<S$, there are
only finitely many $c\in \Dbar(0,s)$ such that
$f_c$ is postcritically finite,
\item for every $c\in D(0,S)$,
$f_c$ is a flexible Latt\`{e}s map, or
\item for every $b,c\in D(0,S)$,
$f_b$ is conjugate to $f_c$.
\end{enumerate}
\end{thm}

As a special case of Theorem~\ref{thm:goodpcf}, we can show that
the unicritical family $z^d+c$ over $\Cp$ has all PCF parameters isolated.
Here is a more precise statement.

\begin{thm}
\label{thm:specialpcf}
%Let $p\geq 2$ be a prime number, and let $m,d$ be integers with $d\geq 2$
%and $d>m\geq 0$. Define
Let $p\geq 2$ be a prime number, and let $d\geq 2$ be an integer. Define
\[ f_c(z) := z^d + c .\]
If $c\in\Cp$ is a parameter for which $f_c$ is postcritically finite, then $|c|_p\leq 1$.
Moreover, for any $a\in\Cp$ with $|a|_p\leq 1$ and for any radius $0<s<1$,
there are only finitely many $c\in \Dbar(a,s)$ such that
$f_c$ is postcritically finite.
\end{thm}

As we will see in Examples~\ref{ex:d2p2} and~\ref{ex:d2p3},
Theorem~\ref{thm:specialpcf} cannot be strengthened to conclude
that there are only finitely many PCF parameters in an open disk $D(a,1)$ of radius 1.
An application of Theorem~\ref{thm:specialpcf}, towards a result on the integrality of PCF
parameters in the unicritical family, will appear in \cite{BenIh2}.

The outline of the paper is as follows.
In Section~\ref{sec:back}, we set notation and recall the essential background we need,
especially on non-archimedean analysis.
In Sections~\ref{sec:triv} and~\ref{sec:lattes}, we prove technical lemmas needed to avoid
degenerations of isotrivial and Latt\`{e}s families.
Section~\ref{sec:orbit} is devoted to the statement and proof of Theorem~\ref{thm:isol},
which is our central tool, using certain $p$-adic dynamical results to study the orbit
of a single marked point in a one-parameter $p$-adic family.
In Section~\ref{sec:cor}, we use Theorem~\ref{thm:isol} and Thurston Rigidity
(stated here as Theorem~\ref{thm:thurston}) to prove 
Theorems~\ref{thm:pcfisol}, \ref{thm:goodpcf}, and~\ref{thm:specialpcf}.
Finally, in Section~\ref{sec:ex}, we present various examples showing that
the conclusions of our theorems cannot be significantly strengthened, and that the
stability conditions cannot be dropped from their hypotheses.

\section{Notation and background}
\label{sec:back}
In this section, we recall some basic facts from non-archimedean analysis,
following \cite[Chapters 2--3]{BenBook}.
Throughout the rest of the paper,
for ease of notation, we fix a prime number $p\geq 2$, and
we denote the absolute value on $\Cp$
simply as $|\cdot|$ rather than $|\cdot|_p$.

\subsection{Disks and power series}
\label{ssec:disks}
Fix $b\in\Cp$ and $R>0$. We write
\[ D(b,R) := \{x\in\Cp \, : \, |x-b|<R \} \quad\text{and}\quad
\Dbar(b,R) := \{x\in\Cp \, : \, |x-b|\leq R \} \]
for the open and closed disks, respectively, centered at $b$ of radius $R$.
 If $R\in |\Cp^{\times}|$, then 
we call $D(b,R)$ a \emph{rational open disk},
and $\Dbar(b,R)$ a \emph{rational closed disk},
which satisfy $D(b,R)\subsetneq \Dbar(b,R)$.
If $R\not\in |\Cp^{\times}|$, then we call
$D(b,R) =  \Dbar(b,R)$ an \emph{irrational disk}.
In spite of the names, all disks in $\Cp$ are open \emph{and} closed
topologically.
In addition, for any $b' \in D(b,R)$, we have $D(b',R)=D(b,R)$.
Similarly, for any $b'\in\Dbar(b,R)$, we have $\Dbar(b',R)=\Dbar(b,R)$.

We also define a disk $D$ in $\PCp=\Cp\cup\{\infty\}$ to be either a disk in $\Cp$
or the complement $\PCp\smallsetminus D'$ of a disk $D'\subseteq\Cp$.
%In the latter case, if $D'$ is rational closed, we say that $D$ is rational open,
%and vice versa; if $D'$ is irrational, then we say $D$ is also irrational.
In the latter case, if $D'$ is rational closed (resp., rational open, irrational),
we say $D$ is rational open (resp., rational closed, irrational).

Let
\[ f(z) = \sum_{n=0}^{\infty} a_n (z-b)^n \in \Cp[[z-b]] \]
be a power series with coefficients in $\Cp$.
Then $f$ converges on the open disk $D(b,R)$ if and only if the sequence
$\{|a_n| r^n \}_{n\geq 0}$ is bounded for every $0<r<R$;
equivalently, if and only if
$\lim_{n\to \infty}|a_n| r^n=0$ for every $0<r<R$.
In that case, the \emph{Weierstrass degree} of $f$ on $D(b,R)$
is the \emph{smallest} integer $m\geq 0$ such that
\begin{equation}
\label{eq:wdegdef}
|a_m|R^m = \max\big\{ |a_n| R^n \, : \, n\geq 0 \big\},
\end{equation}
or $\infty$ if the maximum is not attained.
If the Weierstrass degree of $f-a_0$ is an integer $m\geq 1$,
then $f$ maps $D(b,R)$ onto $D(a_0,S)$, where $S:=|a_m|R^m$;
moreover, every point of $D(a_0,S)$ has exactly $m$ preimages in $D(b,R)$,
counted with multiplicity.

If $R\in|\Cp^{\times}|$, then $f$ converges on the rational closed disk
$\Dbar(b,R)$ if and only if $\lim_{n\to \infty}|a_n| R^n=0$. 
In that case, the Weierstrass degree of $f$ on $\Dbar(0,R)$
is the \emph{greatest} integer $m\geq 0$ satisfying equation~\eqref{eq:wdegdef}.
As with open disks,
$f$ maps $\Dbar(b,R)$ onto $\Dbar(a_0,S)$, where $S=|a_m|R^m$
and $m$ is the Weierstrass degree of $f-a_0$ on $\Dbar(b,R)$;
moreover, every point of $\Dbar(a_0,S)$ has exactly $m$ preimages in $\Dbar(b,R)$,
counted with multiplicity.
In particular, a convergent power series $f$ can have only finitely many zeros
on a rational closed disk \emph{unless} $f$ is the trivial power series,
i.e., is identically zero. Since any open disk is a union of countably many nested
rational closed disks, then, a nontrivial convergent power series can have only
countably many zeros on an open disk.

We will also occasionally need to consider power series $F(z,w)\in\Cp[[z,w]]$
in two variables. The main fact we will need is the following. Writing
such a series as $F(x,y) = \sum_{i,j\geq 0} A_{ij} x^i y^j$, suppose that
$F$ converges for all $(x,y)$ in the bidisk $D(0,r) \times D(0,s)$, with
image contained in the disk $D(0,t)$, with $r,s,t>0$. Then we have
\[ |A_{ij}| r^i s^j \leq t \quad \text{for all } i,j\geq 0,
\,\,\quad \text{ with } |A_{00}|<t . \]

\subsection{Meromorphic functions}
\label{ssec:merom}
If $\phi\in\Cp(z)$ is a rational function of degree $d\geq 1$ with no poles in
a given disk $D\subseteq\Cp$ --- either open or closed --- then $\phi$ has a power
series expansion that converges on $D$. 
Moreover,
the Weierstrass degree of this power series is at most $d$.
(The existence of such a power series holds for the exact same reasons as over $\CC$.
One can add, subtract, and multiply convergent power series, and one can compute
reciprocals of power series with no zeros using the identity
$1/(1-z)=\sum z^n$ for $|z|<1$.)
More generally, we make the following definition.
\begin{defin}
\label{def:merom}
Let $D\subseteq\Cp$ be a disk, either open or closed.
A function $f:D\to\PCp$ is a (\emph{rigid}) \emph{meromorphic function} on $D$
if there exist power series $g,h$ converging on $D$ such that $f=g/h$.
We denote the set of all meromorphic functions on $D$ by $\MM_p(D)$.
\end{defin}
Just as for $\CC$, 
the set $\MM_p(D)$ is a field under addition and multiplication of functions,
and we may speak of zeros and poles of meromorphic functions.
We will also occasionally abuse terminology and refer to the constant function $\infty$
as being meromorphic; that is, we consider every element of the projective line
$\PP^1(\MM_p(D)) = \MM_p(D)\cup\{\infty\}$ to be a meromorphic function on $D$.
A meromorphic function on $D$ that is not identically $0$ or $\infty$ has only finitely many
zeros and poles in any proper subdisk of $D$; in particular, it has only
countably many zeros and poles in $D$, all of which are isolated.
A meromorphic function with no poles in $D$ is in fact given
by a power series converging on $D$.
If two meromorphic functions $f,g\in\MM_p(D)$ agree on an uncountable subset of $D$,
then $f=g$.

Many of the above facts (in both Sections~\ref{ssec:disks} and~\ref{ssec:merom})
are consequences of the Weierstrass Preparation Theorem
(see, for example, \cite[Theorem~14.2]{BenBook}). This result states
that a nontrivial power series $f\in\Cp[[z-b]]$ which converges on a rational closed disk
$D=\Dbar(b,R)$ can be factored uniquely as $f=gh$, where $g\in\Cp[z]$ is
a monic polynomial $g$ with all of its roots in $D$, and where $h\in\Cp[[z-b]]$
converges on $D$ with $|h(z)|$ a nonzero constant on $D$.
(Moreover, the degree of $g$ is the Weierstrass degree of $f$ on $D$.)
We also have the following fact.

\begin{lemma}
\label{lem:smallpoly}
Let $f(z)$ be a nontrivial power series converging on $D(0,S)$, let $0<s<S$,
and let $y_1,\ldots, y_m$ be the zeros of $f$ in $\Dbar(0,s)$, where $0\leq m < \infty$.
%Then for any $0<s<S$, there are only finitely many points $y_1,\ldots, y_m\in \Dbar(0,s)$
%such that $f(y_i)=0$, where $0\leq m < \infty$. Moreover,
Then for any $r>0$, there exists $\eps>0$ so that
\[ |f(x)|\geq\eps \quad \text{for all } \, x\in\Dbar(0,s)\smallsetminus
\bigcup_{i=1}^m D(y_i, r). \]
\end{lemma}

\begin{proof}
Increasing $s$ slightly if necessary, we may assume without loss that
$s\in |\Cp^{\times}|$. The conclusions are then immediate from
the Weierstrass Preparation Theorem.
\end{proof}

\subsection{Norm and distortion}
Let $h(z)=\sum_{n\geq 0} a_n z^n \in\Cp[[z]]$ be a power series converging on $D(0,R)$.
Then for any radius $0<r<R$, we define $\|h\|_{\zeta(0,r)}$ to be the
sup-norm of $h$ on $\Dbar(0,r)$. That is,
\begin{equation}
\label{eq:normzeta}
\|h\|_{\zeta(0,r)} := \sup\big\{ |h(x)| \, : \, x\in\Dbar(0,r) \big\}
= |a_d| r^d,
\end{equation}
where $d$ is the Weierstrass degree of $h$ on $\Dbar(0,r)$.
(The notation $\|\cdot\|_{\zeta(0,r)}$ is meant to evoke the fact that this norm is
a point $\zeta(0,r)$ on the Berkovich disk $D_{\textup{an}}(0,R)$ corresponding
to the disk $\Dbar(0,r)$.)
%, but most of our results here do not require Berkovich's theory.)
If we define a function $L_h:(-\infty,\log R)\to\RR$ by
\[ L_h(\log r) := \log \|h\|_{\zeta(0,r)}, \]
then $L_h$ is continuous and piecewise linear, with nonnegative integer slopes.
Specifically, the slope of $L_h$ just to the right of $\log r$ is the Weierstrass degree
of $h$ on $\Dbar(0,r)$. (The graph of $L_h$ is essentially the Newton copolygon of $h$.)

As in \cite[Section~4]{Ben7}, we also define a quantity $\delta(h,\zeta(0,r))$, which
we call the \emph{distortion} of the power series $h$ on the disk $\Dbar(0,r)$, by
\begin{equation}
\label{eq:distortdef}
\delta(h,\zeta(0,r)) := \log r + \log \|h'\|_{\zeta(0,r)} - \log \|h\|_{\zeta(0,r)} .
\end{equation}
Then $\log r \mapsto \delta(h,\zeta(0,r))$ is again a continuous, real-valued, piecewise
linear function on $(-\infty,\log R)$ with (possibly negative) integer slopes,
since both $h$ and $h'$ are power series converging on $D(0,R)$.
Indeed, if $h'$ and $h$ have Weierstrass degrees $\ell$ and $m$,
respectively, on $\Dbar(0,r)$, then the slope of this function just to the right of $\log r$ is
$1 + \ell - m$.
%In addition, as in \cite[Lemma~4.2]{Ben7}, we have $\delta(h,\zeta(0,r))\leq 1$
%for all $0<r<R$.

%\subsection{Rational functions and dynamics}
%Let $\phi\in\Cp(z)$ be a rational function of degree $d\geq 1$.
%By the Riemann-Hurwitz formula, $\phi$ has exactly $2d-2$ ramification points
%in $\PCp$, counted with appropriate multiplicity.
%
%If $\phi$ has no poles in
%a given disk $D\subseteq\Cp$ --- either open or closed --- then $\phi$ has a power
%series expansion that converges on $D$. Moreover,
%the Weierstrass degree of this power series is at most $d$.
%

\section{Isotrivial families}
\label{sec:triv}
As noted in the introduction,
a family $f_c$ of rational maps is said to be \emph{isotrivial} if it is conjugate
to a constant map $g$. That is, $g\in\Cp(z)$ is a rational function of degree $d$,
and $f_c=h_c^{-1}\circ g\circ h_c$, where $h_c$ is a one-parameter family
of linear fractional transformations. We will need the following lemma to identify
isotrivial families when they arise in proving our main results.

\begin{lemma}
\label{lem:triv}
Let $S>0$,
let $g(z)\in\Cp(z)$ be a rational function of degree $d\geq 2$,
and let $f_c(z)\in\MM_p(D(0,S))(z)$.
%Let $\beta_1(c),\beta_2(c),\beta_3(c)\in\Cp[[c]]$
%be three distinct power series converging on $D(0,S)$,
Let $\beta_1,\beta_2,\beta_3\in\PP^1\big(\MM_p(D(0,S))\big)$
be three distinct meromorphic functions on $D(0,S)$,
and let $\gamma_1,\gamma_2,\gamma_3$ be three distinct points in $\PCp$.
Suppose that for every $c\in D(0,S)$,
$f_c$ is a rational function in $\Cp(z)$ of degree $d$.
Suppose also that for some uncountable subset $W\subseteq D(0,S)$
and for every $c\in W$, there is a linear fractional transformation $h_c\in\PGL(2,\Cp)$
such that
\[ f_c = h_c^{-1}\circ g \circ h_c
\quad\quad\text{and}\quad
h_c(\beta_i(c))=\gamma_i \, \text{ for }i=1,2,3.\]
Then $f_c$ is conjugate to $g$ for \emph{all} $c\in D(0,S)$.
\end{lemma}

\begin{proof}
\textbf{Step 1}.
After a change of coordinates, we may assume without loss that
$\gamma_1=0$, $\gamma_2=\infty$, and $\gamma_3=1$.
We may similarly assume that none of $\beta_1,\beta_2,\beta_3$ is identically $\infty$.

Let $Y$ be the set of parameters $c\in D(0,S)$ for which at least two of
$\beta_1(c)$, $\beta_2(c)$, $\beta_3(c)$ have the same value.
%For $c\in D(0,S)\smallsetminus Y$, define
Define
\begin{equation}
\label{eq:hcdef}
\tilde{h}_c(z) := \frac{\big( z-\beta_1(c)\big)\big( \beta_3(c)-\beta_2(c)\big)}
{\big( z-\beta_2(c)\big)\big( \beta_3(c)-\beta_1(c)\big)}
\in \PGL\big( 2, \MM_p(D(0,S)) \big),
%\in \PGL(2,\Cp),
\end{equation}
and let $\tilde{f}_c:=\tilde{h}_c^{-1}\circ g \circ \tilde{h}_c\in\MM_p(D(0,S))(z)$.
Then for each $c\in D(0,S)\smallsetminus Y$, we have $\tilde{h}_c\in\PGL(2,\Cp)$.
%and hence the map $\tilde{f}_c$ is a rational function in $\Cp(z)$ of degree $d$.
Moreover, for all $c$ in the uncountable set $W\subseteq D(0,S)\smallsetminus Y$,
we have $\tilde{h}_c=h_c$, and hence $\tilde{f}_c=f_c$ for such $c$.
Since the coefficients of both $\tilde{f}_c$ and $f_c$
are meromorphic functions on $D(0,S)$ that agree for all $c\in W$,
we in fact have $\tilde{f}_c=f_c$ as elements of $\MM_p(D(0,S))(z)$,
and hence $\tilde{f}_c=f_c$ for all $c\in D(0,S)$.
Thus, it suffices to show that $Y$ is empty.

\smallskip

\textbf{Step 2}.
Suppose, towards a contradiction, that there is some $c_0\in Y$.
Since the meromorphic functions $\beta_i-\beta_j$ are nontrivial for all $i\neq j$,
all points of $Y$ are isolated.
Translating by $c_0$ in the $c$-variable, we may assume that $c_0=0$.

Let $\LL:=\Cp((c))$, the field of formal Laurent series over $\Cp$,
which is a complete non-archimedean field with respect to the valuation $\ord_0$,
where $\ord_0(\alpha)$ denotes the order of vanishing of $\alpha\in\LL$ at $c=0$.
Let $\calO_{\LL}:=\Cp[[c]]$ be the ring of integers in $\LL$,
and let $\hatL$ be the completion of an algebraic closure of $\LL$.
Then $f_c(z)$ is a rational function in $\LL(z)$ of degree $d$,
and by hypothesis, $f_0(z)$ is a rational function in $\Cp(z)$ of the same degree $d$.
However, $f_0(z)$ is precisely the reduction of $f_c$ with respect to the
valuation $\ord_0$; thus, $f_c\in\LL(z)$ has explicit good reduction with respect to
$\ord_0$, in the sense of Definition~\ref{def:goodred}.

On the other hand, $g(z)\in\LL(z)$ also has explicit good reduction with respect to $\ord_0$
(since all its coefficients are constants in $\Cp$), and $f_c,g \in\LL(z)$ are conjugate
via $\tilde{h}_c\in\PGL(2,\LL)$. By Propositions~8.12 and~8.13 of \cite{BenBook}, then,
we must have $\tilde{h}_c\in\PGL(2,\calO_{\LL})$.
(See also \cite[Th\'{e}or\`{e}me~4]{Riv3}. In the language of Berkovich spaces,
the explicit good reduction of $f_c$ is equivalent to saying that the Gauss point $\zeta(0,1)$
in the Berkovich projective line over $\hatL$ is invariant under $f_c$,
by \cite[Proposition~8.12]{BenBook}.
However, the fact that $g=\tilde{h}_c\circ f_c \circ \tilde{h}_c^{-1}$ has explicit good reduction
says that $\tilde{h}_c^{-1}(\zeta(0,1))$ is also invariant under $f_c$; but there can only
be one such invariant point, by \cite[Proposition~8.13]{BenBook}. Thus,
$\tilde{h}_c(\zeta(0,1))=\zeta(0,1)$, which is equivalent to saying
$\tilde{h}_c\in\PGL(2,\calO_{\LL})$, by \cite[Proposition~4.10(b)]{BenBook}.)

The fact that $\tilde{h}_c\in\PGL(2,\calO_{\LL})$, and hence that
$\tilde{h}_c^{-1}\in\PGL(2,\calO_{\LL})$, implies that
$\beta_1(0)=\tilde{h}_0^{-1}(0)$, $\beta_2(0)=\tilde{h}_0^{-1}(\infty)$, and
$\beta_3(0)=\tilde{h}_0^{-1}(1)$ are all distinct.
Thus, $0\not\in Y$, as desired.
\end{proof}

Observe that the hypothesis that $\deg(f_c)=d$ for all $c\in D(0,S)$ is essential.
For example, consider the family $f_c(z)=cz^2+z$ for $c\in D(0,1)$.
Then for every $c\in D(0,1)\smallsetminus\{0\}$, we have $f_c(z)=c^{-1} g(cz)$,
where $g(z)=z^2+z$. However, $f_0(z)=z$, which is \emph{not} conjugate to $g(z)$,
but which also has strictly smaller degree.

\section{Latt\`{e}s maps}
\label{sec:lattes}

There are a number of equivalent definitions of (flexible) Latt\`{e}s maps in the
literature. Our working definition, which is equivalent to the others by
\cite[Proposition~6.51]{SilADS} and \cite[Proposition~III.1.7(a)]{SilAEC},
is as follows.

\begin{defin}
\label{def:lattes}
Let $f\in\Cp(z)$ be a rational function of degree $d\geq 2$.
We say that $f$ is a \emph{Latt\`{e}s map} if there exist
\begin{itemize}
\item an elliptic curve $E$ defined over $\Cp$,
\item a morphism $\psi:E\to E$, and
\item a finite, separable morphism $q:E\to\PP^1$
\end{itemize}
such that the following diagram commutes
\[ \begin{CD}
E  @>\psi>> E \\
@V{q}VV	 @V{q}VV \\
\PP^1 @>f>> \PP^1
\end{CD} \]
Furthermore, if
we write $E$ in Legendre form $y^2=x(x-1)(x-\lambda)$
with $\lambda\in\Cp\smallsetminus\{0,1\}$,
%$E$ has a Weierstrass equation of the form
%$y^2=x^3+a_2 x^2 + a_4 x + a_6$ over $\Cp$,
and if
\begin{itemize}
\item $\psi$ is given by $\psi(P):=[m]P+T$ for some integer $m\geq 2$
and some $2$-torsion point $T\in E[2]$, and
\item $q=h\circ x$ for some linear fractional transformation $h\in\PGL(2,\Cp)$,
where $x:E\to\PP^1$ is the $x$-coordinate map,
\end{itemize}
then we say $f$ is a \emph{flexible} Latt\`{e}s map.
\end{defin}

In particular, the degree $d$ of a flexible Latt\`{e}s map $f$ is precisely $d=m^2$.
For more on such maps, see \cite{Mil} or \cite[Sections~6.4--6.5]{SilADS}.

In \cite[Corollary~4.8]{Mil}, Milnor proved the following characterization
of Latt\`{e}s maps; he stated it over $\CC$, but because $\CC$ and $\Cp$ are isomorphic as
abstract fields, it also holds in our context. First, any Latt\`{e}s map is PCF,
and its strictly postcritical set consists of either three or four points.
(The strictly postcritical set is, of course,
the union of the strict forward orbits of the critical points.)
Second, a PCF map with exactly four points in its strictly postcritical set
is Latt\`{e}s if and only if all of its critical points are simple
--- that is, they map to their images with multipliclity~2 ---
and none are strictly postcritical.

%Specifically, a rational function
%$f\in\Cp(z)$ is a Latt\`{e}s map if and only if its
%postcritical orbits satisfy the following properties.
%First, all critical points of $f$ are simple; that is, they map to their images with multiplicity~2.
%(Equivalently, $f$ has $2d-2$ distinct critical points, where $d=\deg f$.)
%Second, the union of the \emph{strict} forward orbits of the critical points
%consists of either three or four points.

%Second, the union of the \emph{strict} forward orbits of its critical
%points is precisely the four-element set $h(x(E[2]))$, where $h$ is the linear fractional
%transformation of Definition~\ref{def:lattes}.

In particular, flexible Latt\`{e}s maps have this postcritical structure,
with exactly four postcritical points. Specifically,
the strictly postcritical set is the four-element set $h(x(E[2]))$,
where $h$ is the linear fractional transformation of Definition~\ref{def:lattes}.
Indeed, after changing coordinates by $h$ so that the commutative diagram
of the definition becomes simply $x\circ\psi = f\circ x$, the critical points of $f$
are precisely the $2d-2$ points of $x(E[2m] \smallsetminus E[2])$,
that is, the $x$-coordinates of the $2m$-torsion points that are not $2$-torsion.
The image of $E[2m]$ under $\psi$ is $E[2]$, and hence the image of
$x(E[2m])$ under $f$ is $x(E[2])$.
%If $m$ is odd and $T=O$, then each of the four points
%$\beta\in x(E(2))$ is fixed, with $f^{-1}(\beta)$ consisting of $\beta$ itself
%and $(d-1)/2$ of the $2d-2$ critical points in $x(E[2m]\smallsetminus E[2])$.
%If $m$ is odd and $T\neq O$, then each of the four points
%$\beta\in x(E(2))$ is $2$-periodic, with $f^{-1}(\beta)$ consisting of $f(\beta)$
%and $(d-1)/2$ of the $2d-2$ critical points in $x(E[2m]\smallsetminus E[2])$.
%Finally, if $m$ is even, then the point $\alpha=x(T)$ is fixed, with
%$f^{-1}(\alpha)$ consisting of the four points of $x(E[2])$ --- which are not critical ---
%and $(d-4)/2$ of the critical points, while each of the three points
%$\beta\in x(E(2))\smallsetminus\{\alpha\}$ has preimage consisting of
%$d/2$ of the critical points.

To prove Theorem~\ref{thm:pcfisol},
we will need the following criterion about families of maps having
exactly four strictly postcritical points and satisfying Milnor's criterion.
% from \cite[Corollary~4.8]{Mil}.

\begin{lemma}
\label{lem:lattes}
Let $S>0$, let $f_c(z)\in\MM_p(D(0,S))(z)$, and let 
$\alpha_1,\ldots,\alpha_{2d-2}\in\PP^1\big(\MM_p(D(0,S))\big)$.
Suppose, for each $c\in D(0,S)$, that $f_c$ is a rational function
of degree $d\geq 2$ with critical points $\alpha_1(c),\ldots,\alpha_{2d-2}(c)$,
repeated according to multiplicity.
%Let $d$, $S$, $f_c(z)=\Phi(c,z)$, and $\alpha_1,\ldots,\alpha_{2d-2}$ be
%as in Theorem~\ref{thm:pcfisol}.
%Suppose that there are uncountably many
Suppose also that there are uncountably many
parameters $c\in D(0,S)$ for which the critical points
%$\alpha_1(c),\ldots,\alpha_{2d-2}(c)$
are all distinct, and for which the strictly postcritical set consists of exactly four points,
none of which are critical.

Then either
\begin{enumerate}
\item for all $c\in D(0,S)$, $f_c$ is is a flexible Latt\`{e}s map, or
\item for every $b,c\in D(0,S)$, $f_b$ is conjugate to $f_c$.
\end{enumerate}
\end{lemma}

\begin{proof}
\textbf{Step 1}. Let $W\subseteq D(0,S)$ be the uncountable set
of parameters $c$ satisfying the given restrictions on the orbits of the critical points.
By Milnor's result in \cite[Corollary~4.8]{Mil}, the map $f_c$
is Latt\`{e}s for every $c\in W$. On the other hand, there are only
countably many conjugacy classes of rigid (i.e., non-flexible) Latt\`{e}s maps in $\Cp(z)$.
Indeed, by the results of \cite[Section~6.5]{SilADS}, any such map arises
from an elliptic curve $E$ with CM (of which there are only countably many isomorphism classes),
an endomorphism of $E$ (of which there are only countably many),
and a finite subgroup of $\Aut(E)$ to quotient by (of which there are only finitely many).

Suppose there are an uncountable subset $W'\subseteq W$ and a rigid Latt\`{e}s map
$g(z)$ such that for each $c\in W'$, there exists $h_c(z)\in\PGL(2,\Cp)$ satisfying
$f_c=h_c^{-1}\circ g \circ h_c$.
%such that for all $c\in W'$,
%the map $f_c(z)$ is conjugate to the same rigid Latt\`{e}s map $g(z)$; that is, 
Let $\gamma_1,\gamma_2,\gamma_3\in\PCp$ be three distinct postcritical points of $g$.
For each $i=1,2,3$, let $\beta_i(c):=h_c^{-1}(\gamma_i)$,
each of which must be of the form
\begin{equation}
\label{eq:betachoice}
\beta_i(c) = f_c^{m_i}\big(\alpha_{j_i}(c)\big) .
\end{equation}
In particular, $\beta_i\in\PP^1\big(\MM_p(D(0,S))\big)$ for each $i=1,2,3$.
%which is defined for all $c\in D(0,S)$, not just $c\in W'$.
(There may be different possible choices for $m_i$ and $j_i$ for different $c\in W'$,
but only finitely many. Thus, shrinking $W'$ if necessary, we may assume
equations~\eqref{eq:betachoice} hold for all $c\in W'$ and all $i=1,2,3$, with
$W'$ still uncountable.)
%After a coordinate change, we may assume that
%$\beta_i(c)\neq\infty$ for all $c\in D(0,S)$, and hence each $\beta_i(c)$ is a power
%series in $\Cp[[c]]$ converging on $D(0,S)$.
The hypotheses of Lemma~\ref{lem:triv} therefore hold, and hence
$f_c$ is conjugate to $g$ for every $c\in D(0,S)$, yielding conclusion~(b).

\smallskip

\textbf{Step 2}. We may assume for the remainder of the proof that $f_c$
is rigid Latt\`{e}s for only countably many $c$; thus, shrinking $W$
if necessary, we may assume that $f_c$ is flexible Latt\`{e}s for all $c$
in the uncountable set $W$.
For all such $c$, the strictly postcritical set of $f_c$ consists of
\emph{four} distinct points of the form
$\beta_i(c)$ as in equation~\eqref{eq:betachoice}, for $i=1,2,3,4$.
After a change of coordinates, we may assume that no $\beta_i$ is identically $\infty$.

For each $c\in W$,
%let $E_c:y^2 = x^3 + a_2(c) x^2 + a_4(c) x + a_6(c)$
%be the associated curve in Weierstrass form,
let $E_c: y^2 = x(x-1)(x-\lambda(c))$ be the associated elliptic curve in Legendre form,
$T_c$ the associated $2$-torsion point,
$m\geq 2$ the associated integer,
and $h_c\in\PGL(2,\Cp)$ the associated linear fractional transformation,
as in Definition~\ref{def:lattes}.
Since $h_c\circ x(E_c[2]) = \{\beta_1(c), \beta_2(c), \beta_3(c), \beta_4(c) \}$,
and $x(E_c([2]))=\{\infty,0,1,\lambda(c)\}$,
we may assume (by reindexing if necessary) that
\begin{equation}
\label{eq:hcgamma}
h_c(\lambda(c))=\beta_1(c), \,\, h_c(0)=\beta_2(c), \,\,
h_c(1)=\beta_3(c), \,\, \text{and} \,\,\,  h_c(\infty)=\beta_4(c)
%h_c(\alpha_1(c))=\beta_1(c), \,\, h_c(\alpha_2(c))=\beta_2(c), \,\,
%h_c(\alpha_3(c))=\beta_3(c), \;\text{and}\;  h_c(\infty)=\beta_4(c),
\end{equation}
%with $\alpha_i(c)\in\MM_p(D(0,S))$, and that $E_c$ may be written as
%\[ E_c: y^2 = \big( x- \alpha_1(c) \big)  \big( x- \alpha_2(c) \big)  \big( x- \alpha_3(c) \big), \]
for uncountably many such $c\in W$.
Shrink the uncountable set $W$ if necessary so that equation~\eqref{eq:hcgamma}
holds for \emph{all} $c\in W$. For such $c$, then, we must have
\begin{equation}
\label{eq:hcinv}
h_c^{-1}(z) := \frac{(z-\beta_2(c))(\beta_3(c) - \beta_4(c))}
{(z-\beta_4(c))(\beta_3(c) - \beta_2(c))},
\end{equation}
and hence
\begin{equation}
\label{eq:lambda}
\lambda(c) = \frac{(\beta_1(c)-\beta_2(c))(\beta_3(c) - \beta_4(c))}
{(\beta_1(c)-\beta_4(c))(\beta_3(c) - \beta_2(c))} \in \MM_p(D(0,S)).
\end{equation}
Note that $\lambda$ cannot be identically equal to any of $0,1,\infty$,
since $E_c$ is an elliptic curve for all $c\in W$.
%Moreover, we have $f_c\circ h_c \circ x = h_c\circ x \circ [m]$ as maps from $E_c$ to $\PP^1$;
%in particular, $d=m^2$.

\smallskip

\textbf{Step 3}.
Let $Y\subseteq D(0,S)\smallsetminus W$ be the set of parameters $c\in D(0,S)$
for which at least two of
the values $\beta_1(c), \beta_2(c), \beta_3(c), \beta_4(c)$ coincide.
Since the meromorphic functions $\beta_i-\beta_j$ are nontrivial for all $i\neq j$,
all points of $Y$ are isolated.
We claim that $Y$ is empty.

Suppose, towards a contradiction, that there is some $c_0\in Y$. Translating in
the $c$-variable, we may assume $c_0=0$.
Reindexing $\beta_1,\ldots,\beta_4$ again if necessary,
with the associated changes to $E_c$, $T_c$, and $h_c$,
we may assume that $\lambda(0)\neq 1,\infty$.
Thus, $\ord_0(\lambda) \geq 0 = \ord_0(\lambda-1)$, where $\ord_0$
denotes the order of vanishing at $c=0$.
To prove our claim, it suffices to show that
$\lambda(0)\neq 0$, i.e., that $\ord_0(\lambda)=0$.

As in the proof of Lemma~\ref{lem:triv},
let $\LL:=\Cp((c))$ with valuation $\ord_0$,
and with ring of integers $\calO_{\LL}:=\Cp[[c]]$.
Also as in that proof, the map $f_c\in\LL(z)$ has explicit good reduction
with respect to $\ord_0$.

If $\ord_0(\lambda)>0$, then $E_0$ is a singular curve. Thus, if we consider
$E_c$ as an elliptic curve over the discretely valued field $\LL$, then
$E_c$ has multiplicative reduction. (See, for example, \cite[Proposition~VII.5.1(b)]{SilAEC}.)
In light of Exercise~10.19(a) and Proposition~8.12 of \cite{BenBook}, then,
the map $f_c$ does \emph{not} have explicit good reduction.
(More precisely, the Julia set of $f_c$ in the Berkovich projective line over $\hatL$,
the completion of an algebraic closure of $\LL$,
is a line segment, containing infinitely many points. However, by
\cite[Proposition~8.12]{BenBook}, the Julia set of a map of explicit good reduction
consists of only one point.)
This contradiction to the previous paragraph implies that $\ord_0(\lambda)=0$,
proving our claim.

\smallskip

\textbf{Step 4}. 
By the claim of Step~3, the points $\beta_1(c),\beta_2(c),\beta_3(c),\beta_4(c)$
are distinct for each $c\in D(0,S)$. Hence, for each $c\in D(0,S)$,
the map $h_c$ described by equations~\eqref{eq:hcgamma} and~\eqref{eq:hcinv}
lies in $\PGL(2,\Cp)$, and $\lambda(c)\neq 0,1,\infty$, where
$\lambda$ is the meromorphic function of equation~\eqref{eq:lambda}.

In addition, for each $c\in W$, the $2$-torsion point $T_c$ of Step~2 must 
be one of the four 2-torsion points of $E_c$, namely $(\lambda(c),0)$,
$(0,0)$, $(1,0)$, or $O$, the point at infinity.
Shrinking the uncountable set $W$ if necessary, we may assume that $T_c$
is always the same one of these four points, for every $c\in W$.
For \emph{every} $c\in D(0,S)$, define $\tilde{T}_c$ to be this torsion point on $E_c$.

For every $c\in D(0,S)$, define $\tilde{f}_c$ to be the flexible Latt\`{e}s map associated
to the curve $E_c: y^2=x(x-1)(x-\lambda(c))$, morphism $\psi:P\mapsto [m]P+\tilde{T}_c$,
and linear fractional transformation $h_c\in\PGL(2,\Cp)$.
Then $f_c, \tilde{f}_c\in \MM_p(D(0,S))(z)$, and $f_c=\tilde{f}_c$ for all $c\in W$.
Since $W$ is uncountable, we have $f_c=\tilde{f}_c$ for all $c\in D(0,S)$.
\end{proof}

\section{The orbit of a marked point}
\label{sec:orbit}
The heart of our proof of Theorem~\ref{thm:pcfisol} is the following
result on the orbit of a single marked point in a family $f_c(z)$.

\begin{thm}
\label{thm:isol}
Let $S>0$,
%let $\Phi(c,z)\in\Cp(c,z)$,
%let $\Phi(c,z)\in\Cp[[c]](z)$,
let $f_c(z):=\Phi(c,z)\in\MM_p(D(0,S))(z)$,
and let $\alpha\in\PP^1\big(\MM_p(D(0,S))\big)$.
%and let $\alpha(c)\in\Cp[[c]]$
%be a power series converging on a disk $D(0,S)$.
Suppose that for each $c\in D(0,S)$,
%\[ f_c(z) := \Phi(c,z) \]
$f_c(z)$ is a rational function in $\Cp(z)$ of degree $d\geq 2$.
%, and that $\alpha(c)$ is a critical point of $f_c$.
Let $N>M\geq 0$ be integers, and
let $U_0,\ldots,U_{N}\subseteq \PCp$ be rational open disks such that
for all $c\in D(0,S)$,
\begin{itemize}
\item $\alpha(c)\in U_0$,
\item $f_c(U_j) \subseteq U_{j+1}$ for all $j=0,\ldots, N-1$, and
\item $U_N\subseteq U_M$.
\end{itemize}
Then either
\begin{enumerate}
\item for any $s\in (0,S)$, there are are only finitely many
$c\in\Dbar(0,s)$ such that $\alpha(c)$
and all critical points of $f_c$ in
$U_{M}\cup \cdots \cup U_{N-1}$
are preperiodic under $f_c$,
or
\item there are integers $n>m\geq 0$ such that
$f_c^n(\alpha(c))=f_c^m(\alpha(c))$ for all $c\in D(0,S)$.
%Moreover, for all integers $j>i\geq 0$ with $i\leq m$ and $j\leq n$,
%and with either $i<m$ or $j<n$,
%we have $f_c^j(\alpha(c))\neq f_c^i(\alpha(c))$ for all
%$c\in D(0,S)\smallsetminus\{0\}$.
\end{enumerate}
\end{thm}

%Conclusion~(b) says that $\alpha(c)$ is preperiodic of a certain precise
%orbit structure that holds for all $c\in D(0,S)\smallsetminus\{0\}$.
%That is, not only do we have $f_c^n(\alpha(c))=f_c^m(\alpha(c))$, but
%the integers $m$ and $n$ are minimal
%for \emph{all} $c\in D(0,S)\smallsetminus\{0\}$.
%However, when $c=0$, it is possible that the orbit of $\alpha(0)$
%under $f_0$ degenerates to a shorter preperiodic orbit structure.

%(Note: in conclusion~(b), the identity
%$f_c^n(\alpha(c))=f_c^m(\alpha(c))$ must also hold for all $c$ in
%the original domain $D(0,S)$.)
%%, although there may be more parameters
%%outside $D(0,S)$ where the orbit degenerates.)

To prove Theorem~\ref{thm:isol}, we will need the following lemma
about dynamics in a disk-shaped attracting basin.
It is a variant of \cite[Theorem~4.1]{BIJL}.

\begin{lemma}
\label{lem:bijlvar}
Let $R>0$, let $\gamma\in D(0,R)$, and let
\[h(z)=\sum_{i\geq \ell} A_i z^i \in\Cp[[z]]
\quad \text{with } \ell\geq 1 \text{ and } A_\ell \neq 0\]
be a nontrivial power series fixing the point $0$,
and converging on $D(0,R)$ with Weierstrass degree $d\geq 1$.
Suppose that $h(D(0,R))\subseteq D(0,R)$, and that
\begin{equation}
\label{eq:cohyp}
|A_\ell| R^{\ell -1} < \min\big\{ |m|^m \, : \, 1 \leq m \leq d \big\}.
\end{equation}
Then at least one of the following three conclusions holds:
\begin{enumerate}
\item $h(\gamma)=0$.
\item $\gamma$ has infinite forward orbit under $h$.
\item $h$ has a critical point in $D(0,R)$
  that has infinite forward orbit under $h$.
\end{enumerate}
\end{lemma}

\begin{proof}
%The fact that $h$ maps $D(0,R)$ into itself means that
%$|A_i|R^i\leq R$ for all $i\geq \ell$.
%In addition, the fact that $|h'(0)|<1$ means that if $\ell=1$,
%then $|A_1| < 1$.
%
Let $r>0$ be the smallest radius such that $h$ has a zero
in the punctured disk $\Dbar(0,r)\smallsetminus\{0\}$, or $r=R$
if $h$ has no zeros in $D(0,R)\smallsetminus\{0\}$.
Then
\[ |A_i| r^i \leq |A_{\ell}| r^{\ell}
\quad \text{for all } i\geq \ell, \]
as otherwise $h$ would have a root in $D(0,r)\smallsetminus\{0\}$.
Therefore,
\begin{equation}
\label{eq:hxsmall}
|h(x)| = |A_{\ell}| \, |x|^{\ell} \quad \text{for all } x\in D(0,r).
\end{equation}
Combined with condition~\eqref{eq:cohyp},
it follows that if $0<|x|<r$, then $0<|h(x)|<|x|$.
In particular, the only preperiodic point of $h$ in $D(0,r)$
is the fixed point $x=0$.

If $h(\gamma)\in D(0,r)$, then either conclusion~(a) holds,
or else $h(\gamma)$ is a nonzero point in $D(0,r)$ and hence
is not preperiodic, yielding conclusion~(b).
Similarly, if $h'$ has any zeros in $D(0,r)\smallsetminus\{0\}$,
then $h$ has a wandering critical point, yielding conclusion~(c).
Hence, we assume for the remainder of the proof that $|h(\gamma)|\geq r$,
and that $h'$ has no roots in $D(0,r)\smallsetminus\{0\}$.
These two assumptions imply that
\[ h\big(\Dbar(0,|\gamma|)\big)\supseteq \Dbar(0,r) \]
(because $h(\Dbar(0,|\gamma|)$ is a disk containing both $h(\gamma)$ and $0$), and
\begin{equation}
\label{eq:hpsize}
|h'(x)|= |\ell A_{\ell}| \, |x|^{\ell -1} \quad \text{for all } x\in D(0,r).
\end{equation}

Since the image of any disk $\Dbar(0,t)$ is a disk $\Dbar(0,u)$,
where $u$ is a continuous and increasing function of $t$, there is
a unique radius $S\in [r,|\gamma|]$ such that $h$
maps $\Dbar(0,S)$ onto $\Dbar(0,r)$.
Let $m$ be the Weierstrass degree of $h$ on $\Dbar(0,S)$.
Because $0$ has $\ell$ preimages at $0$ and at least one other
of absolute value $r$, we have $m\geq \ell + 1$.
Moreover, since $\Dbar(0,S)\subseteq D(0,R)$, we also have $m\leq d$.

Define a function
$L:(-\infty,\log R) \to \RR$ by
\begin{align*}
L(\log t) &= m \log t + m \log \|h'\|_{\zeta(0,t)}
- (m-1) \log \|h\|_{\zeta(0,t)}
\\
&= m\delta(h,\zeta(0,t)) + \log \|h\|_{\zeta(0,t)},
\end{align*}
where the norm $\|\cdot\|_{\zeta(0,t)}$ and distortion $\delta$
are as defined in equations~\eqref{eq:normzeta} and~\eqref{eq:distortdef}.
Then $L$ is a continuous, piecewise linear function
(with all slopes integers).
Equations~\eqref{eq:hxsmall} and~\eqref{eq:hpsize} together imply that
\begin{align*}
L(\log r) & = m\log r +
m \log \big( \big|\ell A_\ell \big| r^{\ell - 1} \big)
- (m-1) \log \big( \big| A_{\ell} \big| r^{\ell} \big)
\\
& = \log \big(\big|A_{\ell}\big| r^{\ell} \big) + m\log |\ell|.
\end{align*}
On the other hand, the distortion $\delta$
satisfies $\delta(h,\zeta(0,S))\geq |m|$,
for example by \cite[Lemma~4.2]{Ben7} or \cite[Lemma~3.3]{BIJL}.
Since $h(\zeta(0,S))=\zeta(0,r)$,
then
\begin{align*}
L(\log S) &= m\delta(h,\zeta(0,S)) + \log\|h\|_{\zeta(0,S)}
\geq m\log|m| + \log r
\\
& > \log\big( |A_{\ell}| R^{\ell -1} \big) + \log r
\geq \log\big( |A_{\ell}| r^{\ell} \big)
\\
& \geq 
\log \big(\big|A_{\ell}\big| r^{\ell} \big) + m\log |\ell|
= L(\log r),
\end{align*}
where the second inequality is by hypothesis~\eqref{eq:cohyp}.
%Thus, there must be some $\rho\in (r,R)$
Thus, there must be some $\rho\in (r,S)$
such that $L$ has positive slope at $\log\rho$. Define
\begin{align*}
M &= \text{number of roots of } h' \text{ in } \Dbar(0,\rho),
\text{ counted with multiplicity},
%\\
%N &= \text{number of roots of } h \text{ in } \Dbar(0,\rho),
%\text{ counted with multiplicity},
\\
a &= \text{number of roots $y$ of } h' \text{ in } \Dbar(0,\rho)
\text{ with } h(y)\neq 0, \text{ counted with multiplicity, and}
\\
b &= \text{number of distinct roots of } h \text{ in } \Dbar(0,\rho).
\end{align*}
Then the slope of $L$ at $\log\rho$ is $m(M+1)-(m-1)m$,
because $\log\|F\|_{\zeta(0,t)}$ increases with respect to $\log t$
with slope equal to the number of zeros of $F$ in $\Dbar(0,t)$,
and because $h$ has $m$ roots in $\Dbar(0,\rho)$.
In addition, since a root of $h$ of multiplicity $n\geq 2$
is also a root of $h'$ of multiplicity $n-1$, we have
$M-a = m-b$. Furthermore, we have $b\geq 2$ by our choice of $r$,
because $h$ has at least two distinct roots in $\Dbar(0,\rho)$, at $0$ and at
some $x$ with $|x|=r$. Thus, since the slope 
of $L$ at $\log\rho$ is positive, we have
\[ 0 < m(M+1) - (m-1)m = m(m+a-b+1) - m^2 + m
=m(a -b + 2) \leq ma, \]
and hence $a>0$. That is, there is some critical point $\beta$
of $h$ in $\Dbar(0,\rho)$ with $h(\beta)\neq 0$.
Thus, $h(\beta)\in D(0,r)\smallsetminus\{0\}$,
whence $\beta$ is wandering.
\end{proof}

\begin{proof}[Proof of Theorem~\ref{thm:isol}]
After a change of coordinates, we may assume
that the open disk $U_M$ is $U_M=D(0,R)$ for some $R>0$.
Thus, $f_c^M(\alpha(c))\in U_M=D(0,R)$ for every $c\in D(0,S)$,
and $f_c^{N-M}(z)\in U_N\subseteq D(0,R)$ for every $(c,z)\in D(0,S)\times D(0,R)$.

Define $\Psi:D(0,S)\times D(0,R) \to D(0,R)$ by
\[ \Psi_c(z) = \Psi(c,z) := f_c^{N-M}\Big(z + f_c^M\big(\alpha(c)\big)\Big)
- f_c^M\big(\alpha(c)\big) \in \Cp[[c,z]]. \]
%and write $G_c(z):=\Psi(c,z)$.
Expanded as a power series, we have
\[\Psi_c(z) = \Psi(c,z) = \sum_{i,j\geq 0} A_{i,j} c^i z^j\]
with $A_{i,j}\in\Cp$ satisfying
\begin{equation}
\label{eq:ASRbound}
|A_{i,j}|S^i R^j\leq R \quad \text{for all } i,j\geq 0,
\end{equation}
and with $|A_{0,0}|<R$.

For $i=0$ and $j=1$, we have $|A_{0,1}|\leq 1$.
We consider two cases: that $|A_{0,1}| <1$, and that $|A_{0,1}| =1$.

\smallskip
\smallskip

\textbf{Case 1: Attracting}.
Suppose that $|A_{0,1}| < 1$. 
We will show that there is a power series $\beta(c)\in\Cp[[c]]$
such that $\beta(c)$ is an attracting fixed point of $\Psi_c$ for all $c\in D(0,S)$.
(That is, $\Psi_c(\beta(c))=\beta(c)$, and $|\Psi'_c(\beta(c))|<1$.)
Then, for each $s\in (0,S)$, we will construct a finite set
$\calE_s\subseteq\Dbar(0,s)$ such that either
\begin{itemize}
\item there exists $n\geq 1$ such that for all $c\in\Dbar(0,s)$, we have $\Psi_c^n(0)=\beta(c)$, or
\item for all $c\in \Dbar(0,s)\smallsetminus\calE_s$, either $0$ or a critical point of $\Psi_c$ in $D(0,R)$
has infinite forward orbit under $\Psi_c$,
\end{itemize}
from which the theorem will follow.

\smallskip
\smallskip

\textbf{Step 1}.
For any radius $s\in [0,S)$, define
%For any $c\in D(0,S)$ and any radius $r\in [0,R)$, define
\begin{equation}
\label{eq:tcr}
%t_{c,r}:=
t_{s}:=
\max \bigg\{ \frac{s}{S}, \frac{|A_{0,0}|}{R}, |A_{0,1}| \bigg\} < 1.
\end{equation}

\begin{claim}
\label{cl:attrclaim}
%For $t_{c,r}$ as in equation~\eqref{eq:tcr}
%and all $x,y\in\Dbar(0,t_{c,r}R)$, we have
%For $t_s$ as in equation~\eqref{eq:tcr},
For all $c\in\Dbar(0,s)$
%and all $x,y\in\Dbar(0,t_{r,s}R)$, we have
and all $x,y\in\Dbar(0,t_s R)$, we have
\[ \Psi_c(x)\in\Dbar(0,t_{s}R) \quad\text{and}\quad
\big| \Psi_c(x) - \Psi_c(y) \big| \leq t_{s} |x-y|. \]
%\begin{equation}
%\label{eq:attrclaim}
%\[ \Psi_c(x)\in\Dbar(0,t_{c,r}R) \quad\text{and}\quad
%\big| \Psi_c(x) - \Psi_c(y) \big| \leq t_{c,r} |x-y|. \]
%\end{equation}
\end{claim}

To prove Claim~\ref{cl:attrclaim}, observe that for any
such $c,x,y$, and for any $i,j\geq 0$, we have
% not both $0$,
%\begin{equation}
%\label{eq:clbound}
\[ \big|A_{i,j} c^i x^j\big| \leq 
\bigg( \frac{s}{S} \bigg)^i t_s^j \big( |A_{i,j}| S^i R^j \big) \leq t_{s}^{i+j} R, \]
%\leq t_{s}^{i+j} R \leq t_{s}R, \]
%\bigg( \frac{|c|}{S} \bigg)^i \bigg( \frac{r}{R}\bigg)^j \big( |A_{i,j}| S^i R^j \big)
%\leq t_{c,r}^{i+j} R \leq t_{c,r}R, \]
%\end{equation}
using inequality~\eqref{eq:ASRbound}.
Thus, $|A_{i,j} c^i x^j| \leq t_s R$ if $i+j\geq1$;
%In addition, $|A_{0,0}|\leq t_{c,r}R$ by definition of $t_{c,r}$.
moreover, $|A_{0,0}|\leq t_{s}R$ by definition of $t_{s}$.
The first conclusion of the claim is then immediate.
For the second,
\begin{align*}
\big| \Psi_c(x) - \Psi_c(y) \big| &=
\bigg| \sum_{i,j\geq 0} A_{i,j} c^i (x^j - y^j)\bigg|
\\
& \leq |x-y| \max_{i\geq 0, j\geq 1}
\big\{ \big|A_{i,j} c^i (x^{j-1} + x^{j-2}y + \cdots + y^{j-1})\big| \big\}
\\
& \leq |x-y| \max_{i\geq 0, j\geq 1}
%\bigg\{ \bigg(\frac{|c|}{S}\bigg)^i \bigg( \frac{r}{R}\bigg)^{j-1}
\bigg\{ \bigg(\frac{s}{S}\bigg)^i t_s^{j-1}
\big( |A_{i,j}| S^i R^{j-1}\big) \bigg\}
%\leq t_{c,r}|x-y|
\leq t_{s}|x-y| ,
\end{align*}
where the final inequality is because for $i+j\geq 2$, 
inequality~\eqref{eq:ASRbound} yields
%\[ \bigg(\frac{|c|}{S}\bigg)^i \bigg( \frac{r}{R}\bigg)^{j-1} \big( |A_{i,j}| S^i R^j\big) 
%\leq t_{c,r}^{i+j-1} R \leq t_{c,r} R
\[ \bigg(\frac{s}{S}\bigg)^i t_s^{j-1} \big( |A_{i,j}| S^i R^{j-1}\big) 
\leq t_{s}^{i+j-1} \leq t_{s}, \]
%\quad \text{for } i+j\geq 2, \]
%and because $|A_{0,1}|\leq t_{c,r}$.
and because $|A_{0,1}|\leq t_{s}$ for $i=0$ and $j=1$.
Thus, we have proven Claim~\ref{cl:attrclaim}.

\smallskip
\smallskip

\textbf{Step 2}.
To prove the existence of the fixed point $\beta(c)$,
for each $n\geq 0$, define
\[ g_n(c) \in \Cp[[c]]
\quad \text{by} \quad g_n(c) := \Psi_c^n(0). \]
Then each $g_n$ is a convergent power series on $D(0,S)$, with image in $D(0,R)$.
Moreover, for each $c\in D(0,S)$, by repeated application of Claim~\ref{cl:attrclaim},
we see that $\{g_n(c)\}_{n\geq 0}$ is a Cauchy sequence
of points in $D(0,R)$, and hence that it converges to some point
$\beta(c)$; furthermore, by Claim~\ref{cl:attrclaim} again, this convergence is uniform
in $c\in\Dbar(0,s)$. Because each $g_n$ is a power series
converging on $D(0,S)$ with image in $D(0,R)$, the limit function
$\beta(c)$ is also such a power series.
By construction, we have $\Psi_c(\beta(c))=\beta(c)$, as desired.

By Claim~\ref{cl:attrclaim} yet again, for any $s\in (0,S)$ and any $c\in\Dbar(0,s)$, we have
%By the second conclusion of Claim~\ref{cl:attrclaim} with $r=|\beta(c)|<R$, we have
\begin{equation}
\label{eq:tcrmult}
%|\Psi'_c\big(\beta(c)\big)| \leq t_{c,r} < 1 ,
\beta(c)\in\Dbar(0,t_s R), \quad\text{and}\quad
|\Psi'_c\big(\beta(c)\big)| \leq t_{s} < 1 .
\end{equation}
In particular, since its multiplier is $\Psi'_c\big(\beta(c)\big)$,
we see that $\beta(c)$ is an \emph{attracting} fixed point  of $\Psi_c$.

\smallskip
\smallskip

\textbf{Step 3}.
Changing coordinates to move $\beta(c)$ to the origin, define
\[ H_c(z) = H(c,z) := \Psi\big(c, z + \beta(c)\big) - \beta(c) \in \Cp[[c,z]], \]
which converges on $D(0,S)\times D(0,R)$ with
image in $D(0,R)$ and with $H(c,0)=0$.
Write
\[ H_c(z) = H(c,z) = \sum_{j\geq \ell} \sum_{i\geq 0} B_{i,j} c^i z^j, \]
where $\ell\geq 0$ is the smallest integer such that some
coefficient $B_{i,\ell}$ is nonzero.
(Such $\ell$ exists since the original function $f_c(z)$ is a nonconstant
rational function of $z$.)
In fact, we have $\ell\geq 1$, since $H(c,0)=0$.
Furthermore,
\begin{equation}
\label{eq:HBbound}
|B_{i,j}|S^i R^j \leq R
\quad \text{for all }i\geq 0 \text{ and }j\geq \ell,
\end{equation}
since the image of $H$ is contained in $D(0,R)$.
%%Let $k\geq 0$ be the smallest integer such that $B_{k,\ell}\neq 0$.
Note also that for any $c\in D(0,S)$, the Weierstrass degree of
$H_c$ on $D(0,R)$ is at most $\deg(f_c^{N-M}) = d^{N-M}$.

For each radius $s\in (0,S)$, define
\[ \calE_s := \bigg\{ y\in\Dbar(0,s) \, : \,
\sum_{i\geq 0} B_{i,\ell} y^i = 0 \bigg\}. \]
Note that $\calE_s$ is finite, because it is the set of zeros 
of a nontrivial power series on a disk $\Dbar(0,s)$
which is strictly contained in a larger disk $D(0,S)$ on which the series converges.
%on a disk of a nontrivial power series that converges on a larger disk.
For each $y\in\calE_s$, we may choose a radius $\sigma(y)>0$ such that
\[ 0 < \bigg| \sum_{i\geq 0} B_{i,\ell} c^i \bigg| <
R^{1-\ell} \min\big\{ |m|^m \, : \, 1\leq m\leq d^{N-M}\big\}
\quad \text{for all }
c\in \Dbar\big(y,\sigma(y)\big)\smallsetminus\{y\}. \]
Hence, for any $c\in \Dbar(y,\sigma(y))\smallsetminus \{y\}$,
if we set $h(z)=H_c(z)$ and $\gamma:=-\beta(c)$, then
%either $c=y$, or else
one of the three conclusions of Lemma~\ref{lem:bijlvar} holds.
That is, either
\begin{itemize}
%\item $c\in\calE_s$, or 
\item $H_c(-\beta(c))=0$, i.e., $\Psi_c(0)=\beta(c)$, or
% and hence $f_c^{2N-M}(\alpha(c)) = f_c^{N}(\alpha(c))$, or
\item $z=0$ has infinite forward orbit under $\Psi_c$, or 
%$\alpha(c)$ has infinite forward orbit under $f_c$, or
\item $\Psi_c$ has a critical point in $D(0,R)$ with infinite forward orbit under $\Psi_c$.
%\item $f_c$ has a critical point in
%$U_M\cup\cdots\cup U_N$ that has infinite forward orbit under $f_c$.
\end{itemize}

\smallskip
\smallskip

\textbf{Step 4}. It remains to consider parameters $c$ in the set
\[ W:= \Dbar(0,s) \smallsetminus \bigg(
\bigcup_{y\in\calE_s} \Dbar(y,\sigma(y)) \bigg).\]
By Lemma~\ref{lem:smallpoly} and the definition of $\calE_s$, there is some $\eps>0$ such
\[ C_c \geq \eps \quad\text{for all } c\in W,
\quad \text{where } C_c := \bigg|\sum_{i\geq 0} B_{i,\ell}c^i\bigg| .\]
Let $r:=t_s R$, and 
recall from equation~\eqref{eq:tcrmult} that
the power series $\beta$ maps $\Dbar(0,s)$ into $\Dbar(0,r)\subsetneq D(0,R)$.
%and that
%\[ |B_{0,1}| = |H'_0(0)| = \big| \Psi'_0\big(\beta(0)\big)\big| \leq t_s < 1. \]
%proper subdisk of $D(0,R)$, and hence there is some $r\in (0,R)$
%such that $|\beta(c)|\leq r$ for all $c\in\Dbar(0,s)$.
%Define
%\[ t:=\max\bigg\{ \frac{s}{S}, \frac{r}{R}, |B_{0,1}| \bigg\} < 1.\]
%
%(The fact that $|B_{0,1}|<1$ follows from our assumption
%that $|A_{0,1}|<1$, since
%\[ B_{0,1} = H'_0(0) = \Psi'_0\big(\beta(0)\big)
%= \sum_{j\geq 1} A_{0,j} \big(\beta(0)\big)^{j-1}, \]
%and since $|A_{0,j}\beta(0)^{j-1}| < |A_{0,j}|R^{j-1}\leq 1$ for $j\geq 2$.)
%Then for all $(c,z) \in \Dbar(0,s)\times\Dbar(0,r R)$, we have
Then for all $(c,z) \in \Dbar(0,s)\times\Dbar(0,r)$, we have
\[ \big|H_c(z)\big| = \big| \Psi_c\big( z+ \beta(c) \big) - \Psi_c\big(\beta(c)\big) \big|
\leq t_s |z|, \]
by Claim~\ref{cl:attrclaim} and the fact that $\Psi_c(\beta(c))=\beta(c)$.
%and for all $i\geq 0$ and $j\geq \ell $, we have
%\[ |B_{i,j} c^i z^j| \leq |B_{i,j}| S^i R^{j-1} \bigg(\frac{s}{S}\bigg)^i
%\bigg(\frac{r}{R}\bigg)^{j-1} |z| \leq t|z|, \]
%by inequality~\eqref{eq:HBbound}.
%\[ |H(c,z)| \leq t |z|
%\quad \text{for all } (c,z) \in \Dbar(0,s)\times\Dbar(0,r). \]
Pick $n\geq 1$ large enough that $t_s^n < \eps r^{\ell -1}$.
Then $H_c^n(-\beta(c))\in D(0,\eps r^{\ell})$ for any $c\in W$.

Furthermore, we claim that
\begin{equation}
\label{eq:shrinkcont}
|H_c(z)| = C_c |z|^{\ell}
\quad\text{for any }
(c,z)\in W\times D(0,\eps r^{\ell}).
\end{equation}
Indeed, for any such $(c,z)$ with $z\neq 0$,
and for any $i\geq 0$ and $j\geq\ell+1$, we have
\[ 
\big| B_{i,j} c^i z^j \big|
\leq \big| B_{i,j} \big| S^i R^{j-1} \bigg(\frac{|z|}{R} \bigg)^{j-1} |z|
\leq \bigg(\frac{|z|}{R} \bigg)^{\ell} |z|
\leq \frac{|z|}{r^{\ell}} |z|^{\ell}
< \eps |z|^{\ell} \leq C_c |z|^{\ell},
\]
using inequality~\eqref{eq:HBbound}.
Thus, the $z^{\ell}$ term in the power series for $H_c(z)$ has strictly
larger absolute value than all the other terms, yielding
equation~\eqref{eq:shrinkcont}, as claimed.

In addition, we have $0<C_c |z|^{\ell}< |z|$ for any such $(c,z)$,
because $C_c R^\ell \leq R$, by inequality~\eqref{eq:HBbound} again.
That is, by equation~\eqref{eq:shrinkcont}, we have $0< |H_c(z)| < |z|$.
Thus, for any $c\in W$, since we have $H_c^n(-\beta(c))\in D(0,\eps r^{\ell})$,
it follows either that $H_c^n(-\beta(c))=0$,
or else that $-\beta(c)$ has infinite forward orbit under $H_c$.
Hence, either
\begin{itemize}
\item $\Psi_c^n(0)=\beta(c)$, or
%and hence $f_c^{(n+1)N-nM}(\alpha(c)) = f_c^{nN-(n-1)M}(\alpha(c))$, or
\item $z=0$ has infinite forward orbit under $\Psi_c$.
%\item $\alpha(c)$ has infinite forward orbit under $f_c$.
\end{itemize}

\smallskip

\textbf{Step 5}.
For any $s\in (0,S)$, let $\calE_s$ be the finite set from Step~3,
and let $n\geq 1$ be the integer chosen in Step~4 (so that $t_s^n < \eps r^{\ell}$).
By the bullet lists at the end of Steps~3 and~4, for any $c\in\Dbar(0,s)\smallsetminus\calE_s$,
we have either
\begin{enumerate}
\item $\Psi_c^n(0)=\beta(c)$, or
\item $z=0$ has infinite forward orbit under $\Psi_c$, or 
\item $\Psi_c$ has a critical point in $D(0,R)$ with infinite forward orbit under $\Psi_c$.
\end{enumerate}
If statement~(a) occurs for infinitely many parameters $c\in\Dbar(0,s)$, then the power series
$\Psi_c^n(0)-\beta(c) \in \Cp[[c]]$ must be trivial, since it converges on the strictly larger disk $D(0,S)$.
Thus, we would have $\Psi_c^n(0)=\Psi_c^{n+1}(0) = \beta(c)$ for all $c\in D(0,S)$, and hence
\[ f_c^{(n+1)N-nM}(\alpha(c)) = f_c^{nN-(n-1)M}(\alpha(c))
\quad \text{for all } c\in D(0,S), \]
yielding conclusion~(b) of Theorem~\ref{thm:isol}.

Otherwise, the set $\calE'_s:=\{ c\in \Dbar(0,s) : \Psi_c^n(0)=\beta(c)\}$ is finite.
For any $c\in\Dbar(0,s)$ outside the finite set $\calE_s\cup\calE'_s$,
either statement~(b) or statement~(c) above occurs.
For a given such $c$, statement~(b) implies that $\alpha(c)$ has
infinite forward orbit under $f_c$, and statement~(c) implies that
$f_c$ has a critical point in $U_{M}\cup\cdots\cup U_{N-1}$ with infinite forward orbit.
Hence, conclusion~(a) of Theorem~\ref{thm:isol} holds, completing Case~1.

\smallskip
\smallskip

\textbf{Case 2: Indifferent}.
Assume for the remainder of the proof that $|A_{0,1}|=1$.
Then for each $c\in D(0,S)$, the power series $\Psi_c(z)\in\Cp[[z]]$
has Weierstrass degree~1 and
maps $D(0,R)$ bijectively onto itself.

\smallskip
\smallskip

\textbf{Step~1}.
Because the residue field of $\Cp$ is isomorphic to
$\Fpbar$, there is an integer $e\geq 1$ such that $|A_{0,1}^e - 1| < 1$.
Replacing $N$ by $e(N-M)+M$ in the statement of Theorem~\ref{thm:isol}
and in the definition of $\Psi(c,z)=\Psi_c(z)$ at the start of this proof,
we may assume that $e=1$, and hence that $|A_{0,1}-1|<1$.

Given a fixed radius $s\in (0,S)$, choose $\tilde{s} \in (s,S)$, and
pick a radius $r$ with
\[ \max\bigg\{ |A_{0,0}|, \frac{R\tilde{s}}{S} \bigg\} < r < R . \]
Then $|A_{0,0}|< r$, $|A_{0,1} - 1| r < r$, and for all $j\geq 2$,
\[ \big|A_{0,j}\big| r^j = \big|A_{0,j}\big|R^j \Big(\frac{r}{R}\Big)^j
\leq R \Big(\frac{r}{R}\Big)^2 = \Big( \frac{r}{R} \Big) r, \]
via inequality~\eqref{eq:ASRbound}.
Pick a real number $t$ with
\[ \max\bigg\{ \frac{|A_{0,0}|}{r}, |A_{0,1}-1|,
\frac{r}{R}, \frac{R\tilde{s}}{rS} \bigg\} < t < 1. \]
Then by the above bounds, we have
\begin{equation}
\label{eq:A0jbound}
|A_{0,j}|r^j < t r
\quad \text{for all } j\geq 2 \text{ and } j=0,
\quad\text{with } |A_{0,1} - 1| r < t r.
\end{equation}
In addition, for all $i\geq 1$ and all $j\geq 0$,
inequality~\eqref{eq:ASRbound} yields
\begin{equation}
\label{eq:Aijbound}
|A_{i,j}| \tilde{s}^i r^j \leq |A_{i,j}| S^i R^j \bigg(\frac{\tilde{s}}{S}\bigg)^{i}
\leq R \bigg(\frac{\tilde{s}}{S}\bigg) < t r.
\end{equation}
Combining inequalities~\eqref{eq:A0jbound} and~\eqref{eq:Aijbound},
then, we have
\begin{equation}
\label{eq:trbound}
\big| \Psi_c(z) - z \big| < t r
\quad \text{for all } (c,z)\in \Dbar(0,\tilde{s})\times \Dbar(0,r).
\end{equation}

\smallskip
\smallskip

\textbf{Step~2}.
The \emph{iterative logarithm} of $\Psi_c$, defined
in~\cite[D\'{e}finition~3.7]{Riv1}, is the function
\begin{equation}
\label{eq:Hdef}
\Lambda(c,z) := \lim_{n\to\infty} p^{-n} \big( \Psi_c^{p^n}(z) - z \big).
\end{equation}
According to \cite[Lemme~3.11(iii)]{Riv1},
the bound~\eqref{eq:trbound} yields that
%for each $c\in D(0,\tilde{s})$,
%the limit of equation~\eqref{eq:Hdef} converges uniformly on $z\in D(0,r)$.
%More precisely, we have
\[ \big| \Lambda(c,z) - p^{-n} \big( \Psi_c^{p^n}(z) - z \big)\big|
\leq C_{n} r |p|^n \quad \text{for all }
%\leq C_{c,n} r |p|^n \quad \text{for all }
(c,z)\in \Dbar(0,\tilde{s})\times \Dbar(0,r)
\text{ and } n\geq 1, \]
where, as shown in the proof of \cite[Lemme~3.11(iii)]{Riv1},
\begin{equation}
\label{eq:Ccndef}
C_{n} := \max_{k\geq 1}  |k|^{-1}  \Big(t |p|^{-1/(p^n(p-1))}\Big)^k.
%C_{c,n} := \max_{k\geq 1}  |k|^{-1}  \Big(t |p|^{-1/(p^n(p-1))}\Big)^k.
\end{equation}
Note that for all $n$ large enough, the expression inside the maximum
on the right side of equation~\eqref{eq:Ccndef} approaches $0$ as
$k\to\infty$, since $0<t<1$;
thus, $C_{n}$ is defined and finite for all such $n$.
%thus, $C_{c,n}$ is defined and finite for such $n$.
Moreover, the sequence $\{C_n\}$ is decreasing.
Hence, the limit of equation~\eqref{eq:Hdef} in fact converges uniformly 
on $(c,z)\in \Dbar(0,\tilde{s}) \times \Dbar(0,r)$.
Therefore, $\Lambda(c,z)\in\Cp[[c,z]]$
is a power series converging on $\Dbar(0,\tilde{s})\times\Dbar(0,r)$.

In addition,
by \cite[Proposition~3.16]{Riv1}, a point
$(c,z)\in \Dbar(0,\tilde{s})\times \Dbar(0,r)$
is a zero of $\Lambda$ if and only if $z$ is periodic under $\Psi_c$.
Since $\Psi_c:\Dbar(0,r)\to \Dbar(0,r)$ is bijective, this condition is
equivalent to saying that $z$ is preperiodic under $\Psi_c$.

\smallskip

\textbf{Step 3}.
Define $F(c):= \Lambda(c,0)\in\Cp[[c]]$, which is a power series converging
on $\Dbar(0,\tilde{s})$, with zeros precisely at the parameters $c$ for which
$z=0$ is preperiodic under $\Psi_c$, i.e., for which $\alpha(c)$
is preperiodic under $f_c$.
If $F$ is not identically zero, then $F$ has only finitely many
zeros in the strictly smaller disk $\Dbar(0,s)$.
That is, there are only finitely many $c\in\Dbar(0,s)$ for which $\alpha(c)$
is preperiodic under $f_c$, implying conclusion~(a) of Theorem~\ref{thm:isol}.

Otherwise, we have $F=0$, and hence $\alpha(c)$ is preperiodic under $f_c$
for all $c\in \Dbar(0,\tilde{s})$.
There are only countably many choices of integers $n>m\geq 0$, and hence
there must be some such choice of $m$ and $n$ such that
uncountably many points $c\in D(0,S)$ are zeros of the power series
$f_c^n(\alpha(c))-f_c^m(\alpha(c))$.
A power series with uncountably many zeros must be identically zero,
and hence $f_c^n(\alpha(c))=f_c^m(\alpha(c))$ for all $c\in D(0,S)$,
yielding conclusion~(b).
\end{proof}

\section{Proofs of main theorems}
\label{sec:cor}
Our main results are all consequences of Theorem~\ref{thm:isol}, using
Thurston's powerful rigidity theorem, which we state below as Theorem~\ref{thm:thurston},
in the language of critical orbit relations. Fix an integer $d\geq 2$, and consider
a separable rational function $f(z)$ of degree $d$ with marked critical points
$\alpha_1,\ldots,\alpha_{2d-2}$.
%, with repetitions allowed.
A \emph{set of critical orbit relations} for $f$ is a set of equations of the form
$f^m(\alpha_i)=f^n(\alpha_j)$ for integers $m,n\geq 0$ and indices $i,j\in\{1,\ldots, 2d-2\}$,
including at least one equation of the form $f^{n_i}(\alpha_i)=f^{m_i}(\alpha_i)$
for each $\alpha_i$, with $n_i > m_i \geq 0$. Thus, a set of critical orbit relations
specifies a finite forward orbit for each critical point, and it may also include other
restrictions, such as that some of the critical points coincide,
or that, for example, $f^3(\alpha_1)=f^4(\alpha_2)$.
%Here is the key result.

\begin{thm}[Thurston Rigidity]
\label{thm:thurston}
Fix $d\geq 2$, and let $\calC$ be a set of critical orbit relations.
Assume that $\calC$ is not consistent with a flexible Latt\`{e}s map.
Then, up to conjugacy by linear fractional transformations, there are only
finitely many rational functions $f(z)\in\CC(z)$ of degree $d$
satisfying the relations of $\calC$.
\end{thm}

The original form \cite{Thu;82} of Thurston's result is different;
see also \cite[Theorem~1]{DH;93} for the statement and proof.
An explanation of why the original implies Theorem~\ref{thm:thurston} appears
in \cite[Corollary~3.7]{BBLPP}.
%; see also \cite[Theorem~2.2]{BIJL}.
As with Milnor's criterion, the fact that $\CC$ and $\Cp$ are isomorphic as abstract fields
shows that Theorem~6.1 also holds for $\Cp$ in place of $\CC$.

\begin{proof}[Proof of Theorem~\ref{thm:pcfisol}]
For each $i=1,\ldots, 2d-2$,
one of the conclusions of Theorem~\ref{thm:isol} holds for $\alpha_i(c)$.
If there is some such $i$ for which conclusion~(a) of Theorem~\ref{thm:isol} holds,
then conclusion~(a) of Theorem~\ref{thm:pcfisol} holds, and we are done.

Thus, we may assume for the remainder of the proof that 
conclusion~(b) of Theorem~\ref{thm:isol} holds for every $\alpha_i(c)$.
That is, for each $i=1,\ldots, 2d-2$, there exist  integers $n_i>m_i\geq 0$ such that
\begin{equation}
\label{eq:mini}
f_c^{n_i}\big(\alpha_i(c)\big) = f_c^{m_i}\big(\alpha_i(c)\big)
\end{equation}
for all $c\in D(0,S)$.
In addition, for each $i\neq j$, there may be integers $e_{i,j}, e_{j,i}\geq 0$ so that
\begin{equation}
\label{eq:mini2}
f_c^{e_{i,j}}\big(\alpha_i(c)\big) = f_c^{e_{j,i}}\big(\alpha_j(c)\big)
\end{equation}
for uncountably many $c\in D(0,S)$.
%Without loss, for each $i$, the integers $n_i>m_i\geq 0$
%are the smallest for which equation~\eqref{eq:mini} holds for uncountably many $c$,
%and for each relevant $i\neq j$, the integers $e_{i,j},e_{j,i}\geq 0$
%are the smallest for which equation~\eqref{eq:mini2} holds for uncountably many $c$.
Without loss, for each $i$, the integers $n_i>m_i\geq 0$
are the smallest for which equation~\eqref{eq:mini} holds for uncountably many $c$.
(That is, for any other such $n'_i>m'_i\geq 0$, we have $n'_i\geq n_i$
and $m'_i\geq m_i$.) Similarly, for each relevant $i\neq j$,
we may assume that the integers $e_{i,j},e_{j,i}\geq 0$
are the smallest for which equation~\eqref{eq:mini2} holds for uncountably many $c$.

Since the meromorphic functions $f_c^{n_i}(\alpha_i(c)) - f_c^{m_i}(\alpha_i(c))$
and $f_c^{e_{i,j}}(\alpha_i(c)) - f_c^{e_{j,i}}(\alpha_j(c))$ have
uncountably many zeros, they are trivial on $D(0,S)$. Thus,
\begin{equation}
\label{eq:every}
\text{\framebox{the relations
of equations~\eqref{eq:mini} and~\eqref{eq:mini2} hold for \emph{all} $c\in D(0,S)$.}}
\end{equation}
However, for each $i$,
and for each of the finitely many smaller choices of integers $\ell > k \geq 0$,
there could be (at most) countably many parameters $c\in D(0,S)$ for which
$f_c^k(\alpha_i(c))=f_c^{\ell}(\alpha_i(c))$. Similarly, for each $i\neq j$
and each of the finitely many smaller choices of $k,\ell\geq 0$,
there could be (at most) countably many parameters $c\in D(0,S)$ for which
$f_c^k(\alpha_i(c))=f_c^{\ell}(\alpha_j(c))$.

%More precisely, by considering the finitely many relations of
%equations~\eqref{eq:mini} and~\eqref{eq:mini2} one at a time, we may ensure that
%all these equations hold simultaneously for uncountably many $c$.

Suppose that the relations of equations~\eqref{eq:mini} and~\eqref{eq:mini2}
are consistent with a flexible Latt\`{e}s map.
That is, for uncountably many parameters $c\in D(0,S)$,
all $2d-2$ critical points of $f_c$ are distinct,
%(i.e., equation~\eqref{eq:mini2} never has $e_{i,j}=e_{j,i}=0$),
and the strictly postcritical set consists of four points, none of which are critical.
By Lemma~\ref{lem:lattes},
%$f_c$ is flexible Latt\`{e}s for all $c\in D(0,S)$,
either conclusion~(b) or conclusion~(c) of Theorem~\ref{thm:pcfisol} holds,
and we are done.

Otherwise, by Thurston Rigidity (Theorem~\ref{thm:thurston}),
there are only finitely many rational functions,
up to conjugacy, that satisfy the critical orbit relations~\eqref{eq:mini}
and~\eqref{eq:mini2}. Thus, there exist some $g(z)\in\Cp(z)$ and an uncountable
subset $W\subseteq D(0,S)$ such that $f_c$ is conjugate to $g$ for all $c\in W$.
That is, for any $c\in W$, there exists $h_c\in\PGL(2,\Cp)$ such that
$f_c=h_c^{-1} \circ g \circ h_c$.
We consider two cases.

\smallskip

\textbf{Case 1}.
If the forward orbits of the critical points of $g$ together contain at least three distinct points
$\gamma_1, \gamma_2, \gamma_3$, then there are corresponding distinct
meromorphic functions
\[ \beta_i(c) := f_c^{e_i} \big(\alpha_{j_i}(c)\big)
\quad\text{for }i=1,2,3, \]
for some indices $j_i\in\{1,\ldots,2d-2\}$ and some integers $e_i\geq 0$,
such that $h_c(\beta_i(c))=\gamma_i$ for each $c\in W$ and each $i=1,2,3$.
Thus, by Lemma~\ref{lem:triv}, conclusion~(c) of Theorem~\ref{thm:pcfisol} holds;
again, we are done.

\smallskip

\textbf{Case 2}.
Finally, suppose that the union of the forward orbits of the critical points of $g$
consists of at most two points. Since $g$ must have at least two distinct critical points,
this means that it has exactly two,
either with each of them fixed or with each of them mapping to the other.
Thus, after a change of coordinates, we may assume that $g(z)=z^d$
or that $g(z)=1/z^d$, with critical points at $z=0,\infty$.

Since $f_c$ is conjugate to $g$ for uncountably many $c$,
and since equations~\eqref{eq:mini} and~\eqref{eq:mini2} are optimal for such $c$,
then possibly after reindexing, those equations say that
\begin{equation}
\label{eq:power}
\alpha_1(c)=\alpha_2(c)=\cdots =\alpha_{d-1}(c)
\quad\text{and}\quad
\alpha_d(c)=\alpha_{d+1}(c)=\cdots =\alpha_{2d-2}(c)
\end{equation}
with either
%for all $c\in D(0,S)$, with either
%\[ f_c\big(\alpha_1(c)\big)= \alpha_1(c)
%\quad\text{and}\quad f_c\big(\alpha_d(c)\big)= \alpha_d(c) \]
\begin{equation}
\label{eq:fc1}
f_c\big(\alpha_1(c)\big)= \alpha_1(c)
\quad\text{and}\quad f_c\big(\alpha_d(c)\big)= \alpha_d(c)
\end{equation}
if $g(z)=z^d$, or
%\[ f_c\big(\alpha_1(c)\big)= \alpha_d(c)
%\quad\text{and}\quad f_c\big(\alpha_d(c)\big)= \alpha_1(c) \]
\begin{equation}
\label{eq:fc2}
f_c\big(\alpha_1(c)\big)= \alpha_d(c)
\quad\text{and}\quad f_c\big(\alpha_d(c)\big)= \alpha_1(c)
\end{equation}
if $g(z)=1/z^d$.
By comment~\eqref{eq:every}, equations~\eqref{eq:power} hold for every $c\in D(0,S)$,
as do either equations~\eqref{eq:fc1} if $g(z)=z^d$,
or equations~\eqref{eq:fc2} if $g(z)=1/z^d$.
Moreover, we must have $\alpha_1(c)\neq\alpha_d(c)$ for all $c\in D(0,S)$;
indeed, if some parameter $c_0$ had $\alpha_1(c_0)=\alpha_d(c_0)$,
then $\alpha_1(c_0)$ would have more than $d$ preimages under $f_{c_0}$,
counted with multiplicity.

Hence, for any $c\in D(0,S)$, we may change coordinates in the $z$-variable
to move $\alpha_1(c)$ to $0$, and to move $\alpha_d(c)$ to $\infty$.
Therefore, if $g(z)=z^d$, 
equations~\eqref{eq:power} and~\eqref{eq:fc1} together imply that
for every $c\in D(0,S)$, the rational function $f_c$ is conjugate to $z\mapsto z^d$.
Similarly, if $g(z)=1/z^d$, 
equations~\eqref{eq:power} and~\eqref{eq:fc2} together imply that
for every $c\in D(0,S)$, the rational function $f_c$ is conjugate to $z\mapsto 1/z^d$.
Either way, $f_c$ is conjugate to $g$ for every $c\in D(0,S)$,
yielding conclusion~(c) of Theorem~\ref{thm:pcfisol}.
\end{proof}

%Define $\tilde{h}_c(z)\in \PGL\big(2,\MM_p(D(0,S))\big)$ by
%%%%\[ \tilde{h}_c(z) := \begin{cases}
%%%%z-\alpha_1(c) & \text{ if } \alpha_d=\infty \text{ (and hence } \alpha_1\neq \infty),
%%%%\\
%%%%\dfrac{1}{z-\alpha_d(c)} & \text{ if } \alpha_1=\infty \text{ (and hence } \alpha_d\neq \infty),
%%%%\text{ or }
%%%%\\
%%%%\dfrac{z-\alpha_1(c)}{z-\alpha_d(c)} & \text{ otherwise}.
%%%%\end{cases} \]
%\[ \tilde{h}_c(z) := \frac{z-\alpha_1(c)}{z-\alpha_d(c)} \]
%if neither $\alpha_1$ nor $\alpha_d$ is identically $\infty$,
%or $\tilde{h}_c(z):= z-\alpha_1(c)$ if $\alpha_d=\infty$,
%or $\tilde{h}_c(z):= 1/(z-\alpha_d(c))$ if $\alpha_1=\infty$.
%Then for each $c\in D(0,S)$, we have $\tilde{h}_c\in\PGL(2,\Cp)$,
%since $\alpha_1(c)\neq\alpha_d(c)$.
%Defining $\tilde{f}_c := \tilde{h}_c \circ f_c \circ \tilde{h}_c^{-1}$,
%equations~\eqref{eq:fc1} and~\eqref{eq:fc2} therefore imply that
%\[ \tilde{f}_c(z) = \begin{cases}
%b(c) z^d & \text{ if } g(z)=z^d, \text{ or}
%\\
%b(c)/z^d & \text{ if } g(z)=1/z^d,
%\end{cases} \]
%for some $b(c)\in\MM_p(D(0,S))$ with no zeros or poles in $D(0,S)$.
%Either way, $\tilde{f}_c$ is conjugate to $g$ for every $c\in D(0,S)$,
%via a coordinate change of the form $z\mapsto tz$.
%Since $f_c$ is conjugate to $\tilde{f}_c$, it follows that $f_c$ is
%conjugate to $g$ for every $c\in D(0,S)$,
%yielding conclusion~(c) of Theorem~\ref{thm:pcfisol}.
%\end{proof}

\begin{proof}[Proof of Theorem~\ref{thm:goodpcf}]
%After a $\PGL(2,\calO_p)$-change of coordinates in the $z$-variable
%--- for example, by $z\mapsto z/(1+bz)$ for some $b\in\Cp$ with $|b|$ small ---
%we may assume that no $\alpha_i(c)$ is identically equal to $\infty$.
%That is, we may assume that $\alpha_i(c)\in\MM_p(D(0,S))$.
%
For each point $a\in\PCp$, let $W({\bar{a}})$ be
%the rational open disk that is 
the inverse image of the
point $\bar{a}\in\PP^1(\Fpbar)$ under the reduction map $\PCp\to \PP^1(\Fpbar)$.
That is, if $|a|\leq 1$, then $W({\bar{a}})=D(a,1)$, and if $|a|>1$,
then $W({\bar{a}})=\PCp\smallsetminus\Dbar(0,1)$.

For each $i=1,\ldots,2d-2$ and for every $j\geq 0$,
let $U_{i,j}:=W\big(\overline{f_0^j(\alpha_i(0))}\big)$.
By the fourth bullet point of the hypotheses, we have
$\alpha_i(c)\in U_{i,0}$ for every $c\in D(0,S)$.
Furthermore, by the first and second bullet points,
for every $j\geq 0$ and every $c\in D(0,S)$, we have
\begin{align*}
f_c(U_{i,j}) &= f_c\Big( W\Big(\overline{f_0^j(\alpha_i(0))}\Big) \Big)
\subseteq W\Big( \bar{f}_c \Big( \overline{f_0^j(\alpha_i(0))} \Big) \Big)
\\
& =W\Big( \bar{f}_0 \Big( \overline{f_0^j(\alpha_i(0))} \Big) \Big)
=W\Big( \overline{f_0^{j+1}(\alpha_i(0))} \Big) = U_{i,j+1} .
\end{align*}
In addition, because $\overline{f}_0\in\Fpbar(z)$ and
$\overline{\alpha_i(0)}\in\PP^1(\Fpbar)$ must both be defined over
some finite subfield of $\Fpbar$,
there exist some $N_i > M_i \geq 0$ such that
$\bar{f}_0^{N_i} (\overline{\alpha_i(0)}) = \bar{f}_0^{M_i} (\overline{\alpha_i(0)})$.
Hence, $U_{i,{N_i}} = U_{i,M_i}$.

We have verified that the family $f_c$ satisfies all the hypotheses
of Theorem~\ref{thm:pcfisol} holds.
Therefore, one of the three desired conclusions follows.
\end{proof}

\begin{proof}[Proof of Theorem~\ref{thm:specialpcf}]
Counting with multiplicity, the map $f_c(z)=z^d+c$ has $d-1$ critical points at $z=0$,
and the other $d-1$ at the (fixed) point at $\infty$.

For any $c\in\Cp$ with $|c|>1$, and for any $x\in\Cp$ with $|x|>|c|^{1/d}$,
we have $|f_c(x)| = |x|^d > |x|$. Since $|f_c(0)|=|c|> |c|^{1/d}$ for such $c$, it follows that
\[ 0 < |f_c(0)| < |f^2_c(0)| < |f^3_c(0)| < \cdots ,\]
and hence the critical point at $z=0$ is not preperiodic. 
Thus, $f_c$ is not PCF for $|c|>1$, proving the first claim of Theorem~\ref{thm:specialpcf}.

Given any $a\in\Cp$ with $|a|\leq 1$, then for any $c\in D(0,1)$,
define $g_c(z):=z^d + a+c$, and also define
$\alpha_i(c):=0$ for each $i=1,\ldots, d-1$ and $\alpha_i(c):=\infty$ for $i=d,\ldots,2d-2$.
Then for each $c\in D(0,1)$, the map $g_c$ has explicit good reduction, with
$\bar{g}_c(z) = z^d + \bar{a} = \bar{g}_0(z)$, and clearly with
$\overline{\alpha_i(c)}=\overline{\alpha_i(0)}$.
Hence, by Theorem~\ref{thm:goodpcf}, either
every $g_c$ is flexible Latt\`{e}s, or the family is isotrivial, or our desired conclusion holds.

However, $g_c$ is not flexible Latt\`{e}s. Indeed, for $d\geq 3$, there are repeated critical
points, violating Milnor's criterion; and for $d=2$, the degree of $g_c$ is not a square.
That is, the second conclusion of Theorem~\ref{thm:goodpcf} cannot hold.

In addition, we claim that there are only finitely many $c\in D(0,1)$ for which $g_c$
is conjugate to $g_0$. Indeed, if $a=0$, then no other $g_c$ is conjugate to $g_0(z)=z^d$,
since $z^d$ is the only map in the family with all critical points fixed. If $a\neq 0$,
then both $g_0$ and $g_c$ have fixed critical points at $z=\infty$ and non-fixed critical points
at $z=0$. Thus, if $g_c=h^{-1} \circ g_0 \circ h$ for some $h\in\PGL(2,\Cp)$,
then we must have $h(0)=0$ and $h(\infty)=\infty$, so that
$h(z)=bz$ for some $b\in\Cp^{\times}$. Therefore,
\[ z^d + a+c = g_c(z) = b^{-1} g_0(bz) = b^{d-1} z^d + b^{-1}(a+c), \]
so that $b$ must be a $(d-1)$-st root of unity. Hence there are at most $d-1$ choices
of $c$ with $g_c$ conjugate to $g_0$, proving our claim.

Therefore, the family $g_c$ is neither Latt\`{e}s nor isotrivial.
By Theorem~\ref{thm:goodpcf}, then, our desired conclusion holds.
\end{proof}

\section{Examples}
\label{sec:ex}

In this section, we present examples to justify our earlier claims that
our main theorems cannot be significantly strengthened.
In particular, Examples~\ref{ex:d2p2} and~\ref{ex:d2p3} illustrate that
the conclusions of Theorem~\ref{thm:specialpcf} cannot be significantly strengthened
even for the quadratic family $f_c(z)=z^2+c$.
Moreover, Example~\ref{ex:juan} illustrates the necessity of the
stability hypotheses of Theorems~\ref{thm:pcfisol} and~\ref{thm:goodpcf}.

%there are only finitely many PCF parameters in an open unit disk,
%as opposed to in a proper subdisk.

\begin{example}
\label{ex:d2p2}
Consider the family $f_c(z)=z^2+c$ of Theorem~\ref{thm:specialpcf},
with $d=2$ and $p=2$.
Then both $c=0$ and $c=-2$ are PCF parameters of $f_c$ lying in the $2$-adic
disk $D(0,1)$, and all three of $c=-1$ and $c=\pm i$ are PCF parameters in the $2$-adic
disk $D(1,1)$.
More generally, the roots of the polynomial
\[ g_n(c):=f_c^n(0)+f_c^{n-1}(0) \]
are parameters $c$ for which the orbit of the critical point $z=0$ satisfies $f_c^{n+1}(0)=f_c^n(0)$
but no simpler critical orbit relation. However, it is easy to check that
$g_n(c)\equiv c^{2^{n-1}}\pmod{2}$,
and therefore all roots of $g^n$ lie in the $2$-adic disk $D(0,1)$.
For any such root $c$, the map $f_c$ has an attracting fixed point $\beta\in D(0,1)$,
and the critical point $z=0$ lands on $\beta$ after exactly $n$ iterations.
In particular, there are infinitely many PCF parameters in the disk $D(0,1)$.
\end{example}

\begin{example}
\label{ex:d2p3}
Consider the family $f_c(z)=z^2+c$ of Theorem~\ref{thm:specialpcf},
with $d=2$, and this time with $p=3$.
For any parameter $c$ in the $3$-adic disk $D(1,1)\subseteq\CC_3$,
we have $f_c(0)\equiv 1\pmod{3}$ and $f_c^n(0)\equiv -1 \pmod{3}$ for all $n\geq 2$.
Changing coordinates via $c=1+b$ and $z=-1+w$, consider the family
\[ F_b(w):= f_{1+b}(w-1)+1 = w^2 -2w + (b+3)\in\calO_3[[b,w]], \]
where $\calO_3$ is the ring of integers in $\CC_3$.

For each $n\geq 1$, let 
$I_n:= \la 3, w^{n+1} \ra + b \la b,w\ra^{n-1}$, which is an ideal
of the ring $R:=\calO_3[[b,w]]$.
The reduction of $F_b$ modulo the ideal $I_1:=\la 3,b,w^2\ra$ 
is simply $w\mapsto w$.
Inspired by a similar idea in \cite[Proposition~4.1]{Lub},
it is a simple exercise
to check that if $G,H\in R$ are of the form
$G(w)\equiv w+A \,\,\mod{I_{n+1}}$ and $H(w)\equiv w+B \,\,\mod{I_{n+1}}$
with $A,B\in I_n$, then $G\circ H(w) \equiv w+A+B \,\, \mod{I_{n+1}}$.
Therefore, $G^3(w)\equiv w \,\,\mod{I_{n+1}}$.
By induction, it follows that $F_b^{3^n}(w)\equiv w \,\,\mod{I_{n+1}}$ for every $n\geq 0$,
and hence that $F_b^{3^n}(-b)\equiv -b \,\,\mod{\la 3,b^{n+1} \ra}$.
On the other hand, viewed as a polynomial in the two variables $b,w$,
the highest total-degree term of $F_b^{3^n}(w)$ is $w^{D_n}$, where $D_n=2^{3^n}$.
Thus,
\[ h_n(b):= F_b^{3^n}(-b)+b \in \calO_3[b] \]
is a monic polynomial of degree $D_n$
whose reduction modulo $3$ has a zero of order at least $n+2$ at $\bar{b}=0$.
In particular, $h_n$ has at least $n+2$ roots in the disk $D(0,1)\subseteq\CC_3$.

The critical points of $F_b(w)$ are $w=1,\infty$, and we have
$F_b(1)=b+2$, and
\[ F_b^2(1) = F_b(b+2)=F_b(-b)=b^2+3b+3. \]
Thus, every root $b$ of $h_n$ satisfies $F_b^{3^n}(b+2) = -b$,
and hence $F_b^{2+3^n}(1) =F_b^2(1)$.
That is, every such $b$ is a PCF parameter of the family $F_b$.
Since $h_n$ has at least $n+2$ roots in $D(0,1)$, and since these roots
are distinct by \cite[Lemma~4]{Buff}, it follows that
the family $F_b$ has infinitely many PCF parameters $b\in D(0,1)$.
(More precisely, Buff's result says that the polynomial $f_c^{1+3^n}(0)+f_c(0)\in\ZZ[c]$
has no repeated roots; after our change of coordinates, the same is true of $h_n$.)
Therefore, the original family $f_c$ has infinitely many PCF parameters $c\in D(1,1)$.
For any such $c$, the map $f_c$ fixes the disk $D(-1,1)$, which contains periodic
points of period $3^n$ for every $n\geq 0$; and the critical point $z=0$ lands
on one such periodic point after $2$ iterations.
\end{example}

Our final example illustrates that
the stability conditions of Theorems~\ref{thm:pcfisol} and~\ref{thm:goodpcf}
cannot be removed in general.

\begin{example}
\label{ex:juan}
Fix a prime number $p\geq 2$, and let $d=p+1$. Define
\[ f_c(z):=\bigg( c+\frac{p+1}{p} \bigg) \big( pz^{p+1} - (p+1) z^p + 1 \big), \]
which has critical points at $z=0,1,\infty$. Writing $\gamma(c):=c+(p+1)/p$, we have
\[ \infty\overset{p+1}{\longmapsto}\infty , \quad\quad
1 \overset{2}{\longmapsto} 0 \overset{p}{\longmapsto} \gamma(c), \quad\quad
\frac{p+1}{p} \longmapsto \gamma(c) , \]
where the numbers above the arrows indicate ramification indices.
Thus, $f_c$ is PCF if and only if $\gamma(c)$ is preperiodic.
In particular, $\gamma(0)=(p+1)/p$ is a fixed point of $f_0$, which in turn is therefore PCF.
In addition, we have $f_c'(z) = p(p+1)\gamma(c) z^{p-1}(z-1)$, and hence
\[ f'_0\big( \gamma(0) \big) = p(p+1)\gamma(0) \cdot \gamma(0)^{p-1} \cdot \frac{1}{p}
= \frac{(p+1)^{p+1}}{p^p}, \]
which has $p$-adic absolute value greater than~1. 
That is, $\gamma(0)$ is a repelling fixed point of $f_0$.
Since there are other choices of $c$ for which
$\gamma(c)$ is \emph{not} a fixed point of $f_c$,
then in the terminology of \cite[Section~5]{Riv5}, the family of rational maps $f_c$
with marked critical point $z=1$ has a Misiurewicz bifurcation at $c=0$.

By \cite[Key Lemma]{Riv5}, then, there is an infinite sequence of distinct parameters
$c_n\in\Cp$ with $c_n\to 0$ such that $z=1$ is preperiodic under $f_{c_n}$,
and hence such that $f_{c_n}$ is PCF, for every $n$. The intuitive reason such a sequence
exists is because for nonzero parameters $c$ very close to $0$,
the point $z=\gamma(c)$ will be very close to a repelling fixed point $\beta(c)$.
After many iterations, $f_c^m(\gamma(c))$ will be pushed away from $\beta(c)$,
and by a careful choice of such $c$, we can get this orbit to land wherever we like,
for example on some preimage of $\beta(c)$, yielding a new PCF parameter.
\end{example}

\textbf{Acknowledgments}
The first author gratefully acknowledges the support of NSF grant DMS-150176.
The second author gratefully acknowledges the support of Simons Foundation grant 
622375 and the hospitality of the Korea Institute for Advanced Study during his visit.
The authors also thank Dragos Ghioca, Sarah Koch, and Holly Krieger
for helpful discussions.
%about Thurston Rigidity.

%\vspace{1cm}

\end{document}